\documentclass[12pt, letterpaper,reqno,oneside]{amsart}

\usepackage{amsmath,amsthm,amssymb,amsfonts,enumerate,bm}
\usepackage[usenames,dvipsnames]{xcolor}
\usepackage{hyperref, enumitem, bbm}
\hypersetup{
	colorlinks,
	citecolor=Violet,
	linkcolor=NavyBlue,
	urlcolor=Blue}

\makeatletter
\newcommand\mynobreakpar{\par\nobreak\@afterheading}
\makeatother

\usepackage[left=1in, right=1in, top=1in, bottom=1in]{geometry}
\usepackage[T1]{fontenc} \usepackage{lmodern,times}
\usepackage{microtype}
\newtheorem{thm}{Theorem}[section]
\newtheorem{cor}[thm]{Corollary}
\newtheorem{lem}[thm]{Lemma}
\newtheorem{prop}[thm]{Proposition}
\theoremstyle{definition}
\newtheorem{defn}[thm]{Definition}

\newtheorem{ass}[thm]{Assumption}

\theoremstyle{remark}
\newtheorem{rem}[thm]{Remark}
\newtheorem{exa}[thm]{Example}

\numberwithin{equation}{section}

		\DeclareMathOperator*\essinf{essinf}

		\newcommand\tsum{\textstyle\sum\nolimits}
		\newcommand{\tot}{\tfrac{1}{2}} 
		\newcommand{\oo}[1]{\tfrac{1}{#1}}
		
		\newcommand{\abs}[1]{\left| #1 \right|} 
		\newcommand{\ab}[1]{\langle #1 \rangle} 
		
		\newcommand{\set}[1]{\{#1\}} 
		\newcommand{\sets}[2]{\set{#1\,:\,#2}} 
		
		\newcommand{\seq}[1]{\set{#1_n}_{n\in\N}} 

		\newcommand{\norm}[1]{{||#1||}} 
		\newcommand{\Bnorm}[1]{{\Big|\Big|#1\Big|\Big|}} 
		
		\newcommand{\pare}[1]{\left(#1\right)}
		\newcommand{\bra}[1]{\left[#1\right]}

		\newcommand{\prf}[1]{ \{ #1 \}_{t\in [0,T]}}

		\providecommand{\R}{} \renewcommand{\R}{{\mathbb R}}

		\providecommand{\N}{} \renewcommand{\N}{{\mathbb N}}
		
		\newcommand{\PP}{{\mathbb P}}
		\newcommand{\QQ}{{\mathbb Q}}
		
		\newcommand{\EE}{{\mathbb E}}

		\newcommand{\ee}[1]{ \bE \left[ #1 \right] }

		\newcommand{\FFF}{{\mathbb F}}

		\newcommand{\EN}{{\mathcal E}}

		\newcommand{\eps}{\varepsilon}
		\newcommand{\ld}{\lambda}
		\newcommand{\Ld}{\Lambda}
		
		\newcommand{\vp}{\varphi}
		
		\newcommand{\el}{{\mathbb L}} 
		\newcommand{\lzer}{\el^0}
		\newcommand{\lone}{\el^1}
		\newcommand{\ltwo}{\el^2}
		\newcommand{\lpee}{\el^p}

		\newcommand{\linf}{\el^{\infty}}

		\newcommand{\define}[1]{{\textbf{#1}}}

		\newcommand{\upi}[1]{\suppar{#1}{i}}

			\newcounter{notecounter}

			\newcommand{\efor}{\text{ for }}
			
			\newcommand{\eand}{\text{ and }}
			
			\newcommand{\ewhere}{\text{ where }}

			\newcommand{\cd}{c\` adl\` ag } 

			\renewcommand\aa{{\boldsymbol{\alpha}}}
			\newcommand\af{{\boldsymbol{f}}}

			\newcommand\ag{{\boldsymbol{g}}}

			\newcommand\tmu{{\tilde{\mu}}}
			\newcommand\tnu{{\tilde{\nu}}}

			\newcommand\tp{{\tilde{p}}}

			\newcommand\ax{{\boldsymbol{x}}}

			\newcommand\sA{{\mathcal A}}

			\newcommand\sB{{\mathcal B}}

			\newcommand\bD{{\mathbb D}}

			\newcommand\aE{{\boldsymbol{E}}}
			\newcommand\bE{{\mathbb E}}

			\newcommand\sF{{\mathcal F}}

			\newcommand\sM{{\mathcal M}}

			\newcommand\sP{{\mathcal P}}

			\newcommand\sQ{{\mathcal Q}}

			\newcommand\bR{{\mathbb R}}

			\newcommand\sS{{\mathcal S}}

			\newcommand\sT{{\mathcal T}}

			\newcommand\aY{{\boldsymbol{Y}}}
			
			\newcommand\tY{{\tilde{Y}}}

			\newcommand{\bmo}{\textrm{bmo}}
			
			\newcommand{\BMO}{\textrm{BMO}}
			\newcommand{\EBMO}{\textrm{EBMO}}
			\newcommand{\sinf}{\sS^{\infty}}

			\newcommand{\hPP}{\hat{\PP}}

			\newcommand{\abst}[1]{\bm{\vert}#1\bm{\vert}_{\mathrm{2}}}
			
			\newcommand{\heetau}[1]{ \bE^{\hPP}_{\tau} \left[ #1 \right] }
			
			\newcommand{\amu}{\boldsymbol{\mu}}
			\newcommand{\anu}{\boldsymbol{\nu}}

			\newcommand{\tld}{\tilde{\ld}}
			\renewcommand{\tmu}{\tilde{\mu}}
			\renewcommand{\tnu}{\tilde{\nu}}
			
			\newcommand{\tanu}{\tilde{\anu}}

			\newcommand{\upl}[1]{{#1}^{\ld}}

			\renewcommand{\upi}[1]{{#1}^i}

			\newcommand{\Ei}{\upi{E}}

			\newcommand{\pii}{\upi{\pi}}

			\newcommand{\Bl}{\upl{B}}

			\newcommand{\sAl}{\upl{\sA}}

			\newcommand{\om}{\overline{m}}
			\newcommand{\on}{\overline{n}}
			\newcommand{\Yl}{Y^{\ld}}
			
			\newcommand{\bmul}{\overline{\mu}^{\ld}}
			\newcommand{\bnul}{\overline{\nu}^{\ld}}
			\newcommand{\mul}{\mu^{\ld}}
			\newcommand{\nul}{\nu^{\ld}}
			\newcommand{\pil}{\pi^{\ld}}
			\newcommand{\hQl}{\hat{\QQ}^{\ld}}
			\newcommand{\hql}{\hat{q}^{\ld}}
			
			\renewcommand{\Xi}{X^i}
			\newcommand{\Yil}{Y^{i,\ld}}
			\newcommand{\muil}{\mu^{i,\ld}}
			\newcommand{\nuil}{\nu^{i,\ld}}
			\newcommand{\muilt}{\mu^{i,\tld}}
			\newcommand{\nuilt}{\nu^{i,\tld}}

			\newcommand{\cme}{C_{\ref{lem: mu est}}}
			\newcommand{\eme}{\eps_{\ref{lem: mu est}}}
			\newcommand{\cle}{C_{\ref{lem:lambda est}}}
			\newcommand{\ch}{C_{\ref{lem: H}}}
			\renewcommand{\cd}{C_{\ref{lem: D}}}
			\newcommand{\clip}{C_{\ref{lem:67DB}}}
			
			\newcommand{\eed}{\eps_{\ref{lem: excess demand}}}

			\renewcommand{\tanu}{(\tnu^i)_i}

			\newcommand{\anul}{(\nuil)_i}

\title[Incomplete stochastic equilibria]{Incomplete stochastic equilibria for dynamic monetary utility}

\begin{document}

\author[]{Constantinos Kardaras}
\address{Constantinos Kardaras, Department of Statistics,
            London School of Economics and Political Science}
\email{k.kardaras@lse.ac.uk}

\thanks{\ \\[-2ex] \indent\emph{Acknowledgements:}
 The authors would like to thank the referees and AE for their suggestions which help us to improve the paper.
  The authors would like to thank Ulrich Horst, Sergio Pulido, Frank Riedel and Walter Tr\"ockel for valuable conversations. The third author acknowledges
  the support by the National Science Foundation under Grants No. DMS-0706947
  (2010 - 2015) and Grant No.~DMS-1107465 (2012 - 2017). Any opinions,
  findings and conclusions or recommendations expressed in this material are
  those of the author(s) and do not necessarily reflect the views of the
  National Science Foundation (NSF)}

\author[]{Hao Xing}
\address{Hao Xing, Department of Statistics,
            London School of Economics and Political Science}
\email{h.xing@lse.ac.uk}

\author[]{Gordan \v{Z}itkovi\'{c}}
\address{Gordan \v{Z}itkovi\'{c}, Department of Mathematics, University of Texas at Austin}
\email{gordanz@math.utexas.edu}

\date{\today}

\begin{abstract} We study existence and uniqueness of continuous-time
stochastic Radner equilibria in an incomplete market model among a group of agents whose preference is characterized by cash invariant time-consistent monetary utilities. An assumption
of ``smallness'' type is shown to be sufficient for existence and
uniqueness. In particular, this assumption encapsulates settings with small endowments, small time-horizon, or a large population of weakly heterogeneous agents.  
Central role in our analysis is played by a fully-coupled
nonlinear system of quadratic BSDEs.  \end{abstract}


\maketitle

\section*{Introduction}

\subsection*{The equilibrium problem} The focus of the present paper is
the problem of
existence and uniqueness of a competitive (Radner) equilibrium in an
incomplete continuous-time stochastic model of a financial market. A
discrete version of our model was introduced by Radner in \cite{Rad82} as
an extension of the classical Arrow-Debreu framework, with the goal of
understanding how asset prices in financial (or any other) markets are
formed, under minimal assumption on the ingredients or the underlying
market structure. One of those assumptions is often market completeness;
more precisely, it is usually postulated that the range of various types of
transactions the markets allow is such that the wealth distribution among
agents, after all the trading is done, is Pareto optimal, i.e., that no
further redistribution of wealth can make one agent better off without
hurting  somebody else. Real markets are not complete; in fact, as it turns
out, the precise way in which completeness fails matters greatly for the
output and should be understood as an a-priori constraint.  Indeed, it is
instructive to ask the following questions: Why are markets incomplete in
the first place? Would rational economic agents not continue introducing
new assets into the market, as long as it is still useful? The answer is
that they, indeed, would, were it not for exogenously-imposed constraints
out there, no markets exist for most contingencies; those markets that do
exist are heavily regulated, transactions costs are imposed, short selling
is sometimes prohibited, liquidity effects render replication impossible,
etc. Instead of delving into the modeling issues regarding various types of
\emph{completeness constraints}, we point the reader to \cite{Zit12} where
a longer discussion of such issues can be found.

\subsection*{The ``fast-and-slow'' model} The particular setting we
subscribe to here is one of the simplest from the financial point of view.
It, nevertheless, exhibits many of the interesting features found in more
general incomplete structures and admits a straightforward continuous-time
formulation.  It corresponds essentially to the so-called ``fast-and-slow''
completeness constraint, introduced in \cite{Zit12}.

One of the ways in which the ``fast-and-slow'' completeness constraint can
be envisioned is by allowing for different speeds at which information of
two different kinds is incorporated and processed.  The discrete-time
version of the model is described in detail in \cite[p.~213]{MagQui96},
where it  goes under the heading of ``short-lived'' asset models.  Therein,
at each node in the event tree, the agents have access to a number of
short-lived assets, i.e., assets whose life-span ends in one unit of time,
at which time all the dividends are distributed.  The prices of such assets
are determined in the equilibrium, but their number is typically not
sufficient to guarantee local (and therefore global) completeness of the
market. In our, continuous time model, the underlying filtration is
generated by two independent Brownian motions ($B$ and $W$). Positioned
the ``node'' $(\omega,t)$, we think of $dB_t$ and $dW_t$ as two independent
symmetric random variables, realized at time $t+dt$,  with values $\pm
\sqrt {dt}$.  Allowing the agents to insure each other only with respect to
the risks contained in $dB$, we denote the (equilibrium) price of such an
"asset" by $-\ld_t\, dt$.  As already hinted to above, one possible
economic rationale behind this type of constraint is obtained by thinking
of $dB$ as the readily-available (fast) information, while $dW$ models
slower information which will be incorporated into the process $\ld_t$
indirectly, and only at later dates.  For simplicity, we also fix the spot
interest rate to $0$, allowing agents to transfer wealth from $t$ to $t+dt$
costlessly and profitlessly. Since in our setting consumption can occur only at terminal time, the interest rate can be taken exogenously. The normalization of zero interest rate is for expositional simplicity and is commonly used for model without intertemporal consumption, cf, eg. \cite{LV}.

For mathematical convenience, and to be able to access the available
continuous-time results, we concatenate all short-lived assets with payoffs
$dB_t$ and prices $-\ld_t\, dt$ into a single asset $B^{\ld}_t = B_t +
\int_0^t \ld_u\, du$. It should not be thought of as an asset that carries
a dividend at time $T$, but only as a single-object representation of the
family of all infinitesimal, short-lived assets.

As a context for the  ''fast-and-slow'' constraint, we consider a finite
number $I$ of agents; we assume that their preference structure is characterized by a class of dynamic monetary utilities. The notion of dynamic monetary utility is closely related to dynamic risk measure, which is a time-consistent extension of the static risk measures introduced by Artzner et al. \cite{Artzner}; this time-consistency property is in line with the notion introduced by Koopmans \cite{Koopmans} and Duffie and Epstein \cite{Duffie-Epstein}.  Further information on dynamic risk measures can be found in \cite{CDK, CDK2, Barrieu-ElKaroui, DS, KM, Bion-Nadal, CK, Del-Peng-Gianin} among others. Dynamic monetary utility can also be characterized by $g$-expectations, i.e., solutions of a class of Backward Stochastic Differential Equations (BSDE), introduced by Peng \cite{Peng-g}. As was shown by Delbaen, Peng, and Rosazza Gianin \cite{Del-Peng-Gianin} that any dynamic monetary utility can be represented as a $g$-expectation.

In this paper, we consider a class of dynamic monetary utilities which are sandwiched between two entropic monetary utilities. The simplest example of this class is a group of exponential utilities with idiosyncratic risk-aversion parameters. The cash-invariant or monetary property of  the agents' utilities is absolutely crucial for all of our results as it induces a
``backward'' structure to our problem, which, while still very difficult to
analyze, allows us to make a significant step forward.

\subsection*{The representative-agent approach, and its failure in
incomplete markets} The classical and nearly ubiquitous approach to
existence of equilibria in complete markets is using the so-called
re\-pre\-sen\-ta\-tive-agent approach. Here, the agents' endowments are
first aggregated and then split in a Pareto-optimal way. Along the way, a
pricing measure is produced, and then, a-posteriori, a market is
constructed whose unique martingale measure is precisely that particular
pricing measure. As long as no completeness constraints are imposed, this
approach works extremely well, pretty much independently of the shape of
the agents' utility functions (see, e.g.,
\cite{DufHua85,Duf86,KarLakLehShr91,KarLehShr90,KarLehShr91,DanPon92,
AndRai08, Zit06} for a sample of continuous-time literature). A convenient
exposition of some of these and many other  results, together with a
thorough classical literature overview can be found in \cite[Chapter 4, Notes section]{KarShr98}.

The incomplete case requires a completely different approach and what were
once minute details, now become salient features.  The failure of
representative-agent methods under incompleteness are directly related to
the inability of the market to achieve Pareto optimality by wealth
redistribution. Indeed, when not every transaction can be implemented
through the market, one cannot reduce the search for the equilibrium to a
finite-dimensional ``manifold'' of Pareto-optimal allocations. Even more
dramatically, the whole nature of what is considered a solution to the
equilibrium problem changes. In the complete case, one simply needs to
identify a market-clearing valuation measure. In the present
``fast-and-slow'' formulation, the very family of all replicable claims (in
addition to the valuation measure) has to be determined. This significantly
impacts the ``dimensionality'' of the problem and calls for a different
toolbox.

\subsection*{Our probabilistic-analytic approach} The direction of the
present paper is partially similar to that of \cite{Zit12}, where a much
simpler model of the ``fast-and-slow'' type is introduced and considered.
Here, however, the setting is different and somewhat closer to \cite{Zha12}
and \cite{ChoLar14}. The fast  component is modeled by an independent
Brownian motion, instead of the one-jump process. Also, unlike in any of
the above papers, pure PDE techniques are largely replaced or supplemented
by probabilistic ones, and much stronger results are obtained.

Doing away with the Markovian assumption, we allow for a collection of
unbounded random variables, satisfying suitable integrability assumptions,
to act as random endowments and characterize the equilibrium as a
(functional of a) solution to a nonlinear system of quadratic BSDEs. Unlike single quadratic BSDE,
whose theory is by now quite complete (see e.g., \cite{Kobylanski,
BriHu06,BriHu08, DelHuBao10, BriEli13, Barrieu-El-Karoui} for a sample),
the systems of quadratic BSDEs are much less understood. The main
difficulty is that the comparison theorem may fail to hold for BSDE systems
(see \cite{Hu-Peng}).  Moreover, Frei and dos Reis (see
\cite{Frei-dosReis}) constructed a quadratic BSDE system which has bounded
terminal condition but admits no solution. The strongest general-purpose
result seems to be the one of Tevzadze (see \cite{Tevzadze}), which
guarantees existence under an ``$\linf$-smallness'' condition placed on the
terminal conditions.

Like in \cite{Tevzadze}, but unlike in \cite{Zit12} or \cite{ChoLar14}, our
general result imposes no regularity conditions on the agents' random
endowments. Unlike in \cite{Tevzadze}, we allow here for unbounded terminal conditions (random
endowments), and measure their size using an ``entropic'' $\BMO$-type norm
strictly weaker than the $\linf$-norm. Existence of equilibria is established when random endowments have small entropic-$\BMO$-norm. In addition, the equilibrium
constructed is unique in a global sense (as in \cite{Kramkov-Pulido}, where
a different quadratic BSDE system is studied).

One interesting feature of our general result is that it is 
independent of the number of agents (number of equations in the BSDE system). This is different from \cite{Tevzadze} and leads to the following
observation: the equilibrium exists as soon as ``sufficiently many
sufficiently homogeneous'' (under an appropriate notion of homogeneity)
agents share a given total endowment,  which is not assumed to be small.
This is precisely the natural context of a number of competitive
equilibrium models with a large number of small agents, none of whom has a
dominating sway over the price.

Another feature of our general result is its independence of is the time horizon. Indirectly, this leads to the fact that existence and uniqueness also holds when the time horizon is sufficiently small, but the random
endowments are not limited in size. Under the additional assumption of
Malliavin differentiabilty, a lower bound on how small the horizon has to
be to guarantee existence and uniqueness turns out to be inversely
proportional to the size of the (Malliavin) derivatives of random
endowments.  This extends \cite[Theorem 3.1]{ChoLar14} to a non-Markovian
setting. Interestingly, both the $\linf$-smallness of the random endowments
and the smallness of the time-horizon are implied by the
small-entropic-$\BMO$-norm condition mentioned above, and the existence
theorems under these conditions can be seen as special cases of our general
result.

\subsection*{Some notational conventions} As we will be dealing with
various classes of vector-valued random variables and stochastic processes,
we shall introduce sufficiently compact notation to make reading more
palatable.

A time horizon $T>0$ is fixed throughout. An equality sign between random
variables signals almost-sure equality, while one between two processes
signifies Lebesgue-almost everywhere, almost sure equality. Any two processes
that are equal in this sense will be identified; this, in particular, applies
to indistinguishable c\` adl\` ag processes. Let
$(\Omega,\sF_T,\FFF=\prf{\sF_t},\PP)$ be a filtrated probability space, whose filtration $\FFF$ is the augmented filtration generated by two independent Brownian motion $B, W$ and satisfies the usual conditions of completeness and right-continuity. $\sT$
denotes the set of all $[0,T]$-valued $\FFF$-stopping times, and $\sP^2$
denotes the set of all predictable processes $\prf{\mu_t}$ such that $\int_0^T
\mu_t^2\, dt<\infty$, a.s. The integral $\int_0^{\cdot} \mu_u\, d\hat{B}_u$ of
$\mu\in\sP^2$  with respect to an $\FFF$-Brownian motion $\hat{B}$ is alternatively
denoted by $\mu\cdot \hat{B}$, while the stochastic (Dol\' eans-Dade) exponential
retains the standard notation $\EN(\cdot)$. The $\lpee$-spaces, $p\in
[1,\infty]$ are all defined with respect to $(\Omega,\sF_T,\PP)$,
$\mathbb{L}^0$ denotes the set of ($\PP$-equivalence classes) of finite-valued
random variables on this space. For a continuous adapted process $\prf{Y_t}$,
we set \[ \norm{Y}_{\sinf} = \norm{\textstyle\sup_{t\in [0,T]}
\abs{Y_t}}_{\linf}, \] and denote  by $\sinf$ the space of all such $Y$ with
$\norm{Y}_{\sinf}<\infty$. For $p\geq 1$, the space of all
$\mu\in\sP^2$ with $\norm{\mu}_{H^p}^p = \EE\bra{\textstyle\int_0^T
\abs{\mu_u}^p\, du}<\infty$ is denoted by $H^p$, an alias for the Lebesgue
space $\lpee$ on the product $[0,T]\times \Omega$.

Given a probability measure $\hPP$ and a $\hPP$-martingale $M$, we define its
BMO-norm by \[ \norm{M}_{\BMO(\hPP)}^2= \sup_{\tau \in \sT}\Bnorm{
\heetau{ \ab{ M }_T - \ab{M}_{\tau}} }_{\linf}, \] where $\heetau{\cdot} \equiv \EE^{\hPP}[\cdot|\sF_{\tau}]$ denotes conditional expectation  with
respect to $\sF_{\tau}$, computed under $\hPP$.  The set of all
$\hPP$-martingales $M$ with $\norm{M}_{\BMO(\hPP)} < \infty$ is denoted by
$\BMO(\hPP)$, or, simply, $\BMO$, when $\hPP=\PP$. When applied to random
variables, $X\in\BMO(\hPP)$ means that $X=M_T$, for some $M\in\BMO(\hPP)$. In
the same vein, we define (for some, and then any, $(\hPP,\FFF)$-Brownian
motion $\hat{B}$) \[ \bmo(\hPP)= \sets{\mu\in \sP^2}{ \mu \cdot \hat{B} \in \BMO(\hPP)},\]
with the norm $\norm{\mu}_{\bmo(\hPP)} = \norm{\mu\cdot \hat{B}}_{\BMO(\hPP)}$. The
same convention as previously is used: the dependence on $\hPP$ is suppressed when
$\hPP=\PP$.

Many of our objects will take values in $\R^I$, for some fixed $I\in\N$.
Those are typically denoted by bold letters such as $\aE,
\amu,\anu,\aa$, etc. If specific components are needed, they will be given
a superscript - e.g., $\aE=(\Ei)_i$. Unquantified variables $i,j$ always
range over $\set{1,2,\dots, I}$. The topology of $\R^I$ is induced by the
Euclidean norm $\abst{\cdot}$, defined by $\abst{\ax}=\sqrt{\sum_i
\abs{x^i}^2}$ for $\ax = (x^i)_i \in \R^I$.  All standard operations and relations
(including the absolute value $\abs{\cdot}$ and order $\leq$) between
$\R^k$-valued variables are considered componentwise.

\section{Monetary utilities}\label{sec: utility}

\subsection{Relative entropy} \label{subsec:entropy}

We begin with a convention which extends the definition of expectation to $\mathbb{L}^0$. For any random variable $G \in \mathbb{L}^0$ and probability measure $\QQ \sim \PP$, set 
\[
\EE^{\QQ}[G] := \ \downarrow \lim_{n\rightarrow \infty} \EE^{\QQ}[G \vee (-n)] = \left\{\begin{array}{ll} \infty, & \text{if } \EE^{\QQ}[G_+] =\infty, \\ \EE^{\QQ}[G_+] - \EE^{\QQ}[G_-], & \text{otherwise} \end{array}\right.,
\]
where, as usual, $G_+ = \max\{G, 0\}$ and $G_- = \max\{-G, 0\}$. Furthermore, define
\[
H(\QQ|\PP) = \EE^{\PP} \left[ \frac{d\QQ}{d\PP}\log \frac{d\QQ}{d\PP} \right], \quad \QQ \sim \PP,
\]
to be the relative entropy of $\QQ$ with respect to $\PP$, and define also the pair $(p(\QQ),q(\QQ))$ of predictable processes implicitly via the density $d\QQ / d\PP = \mathcal{E}(- p(\QQ) \cdot B - q(\QQ) \cdot W)_T$. Finally, set
\[
\mathcal{Q} = \left\{ \QQ \sim \PP \ \Big| \  \frac{d\QQ}{d\PP} \in \bigcup_{p>1}
\mathbb{L}^p(\PP) \right\}.
\]
The following lemma is similar to a standard result from the
literature, but requires a separate proof due to our use of the nonstandard
dual domain $\sQ$.

\begin{lem} \label{lem:CE-Q}
It holds that
\[
H(\QQ | \PP) = \frac{1}{2} \EE^{\QQ} \left[  \int_0^T (p^2_u(\QQ) + q^2_u(\QQ) ) du \right] < \infty, \quad \QQ \in \mathcal{Q}.
\]	
Furthermore, we have	
\begin{equation}\label{eq: exp CE}
	-\delta \log \EE^{\PP}[e^{-G/\delta}] = \inf_{\QQ\in \mathcal{Q}}
	\big\{\EE^{\QQ}[G] + \delta H(\QQ|\PP)\big\}, \quad \forall \delta > 0, \  G \in \mathbb{L}^{0}.
\end{equation}
\end{lem}

\subsection{Definition and properties}

In the sequel, we shall consider a random field $f: \Omega \times [0, T] \times \R^2 \rightarrow \R_+$ with the following properties.

\begin{ass}\label{ass:f}
The function $f: \Omega \times [0, T] \times \R^2 \rightarrow \R_+$ is such that:
\begin{itemize}
	\item for all $(p,q)\in \R^2$, $f(\cdot, \cdot, p, q)$ is a predictable process;
	\item for all $(\omega, t) \in \Omega \times [0, T]$, $f(\omega, t, 0, 0) = 0$, $f(\omega, t, \cdot, \cdot)$ is $C^2(\mathbb{R}^2)$ with gradient at $(0,0) \in \bR^2$ satisfying $D f(\omega, t, 0,0) = (0,0)$, and there exist constants $0 < \delta \leq \Delta<\infty$  such that 
	both eigenvalues of the
	Hessian $D^2 f (\omega, t, \cdot, \cdot)$ belong to $[\delta,\Delta]$.
\end{itemize}
\end{ass}

In the sequel, and in order to simplify notation, in random fields like $f$, whenever we want to stress dependence of $(p,q) \in \R^2$, we write $f_{\omega, t} (p, q)$ instead of $f(\omega, t, p, q)$, or even $f (p, q)$ when $(\omega, t) \in \Omega \times [0, T]$ is fixed. Furthermore, and as typical in BSDE theory, we shall frequently omit the argument $\omega$, especially in the context where $f$ has to be evaluated at predictable processes $(p_t, q_t)_{t \in [0, T]}$, where we shall simply write $f_t(p_t, q_t)$.

Any $f$ satisfying Assumption \ref{ass:f} is such that $f_{\omega, t}(\cdot, \cdot)$ is clearly nonnegative and strictly convex for all $(\omega, t) \in \Omega \times [0, T]$. Taylor's theorem implies that 
\begin{align}
\label{equ:sand} 
\tot \delta (p^2+q^2) \leq f_{\omega, t} (p,q) \leq \tot
\Delta (p^2+q^2),\text{ for all } (\omega, t, p, q) \in \Omega \times [0,T] \times \R^2.
\end{align}

Whenever $f$ satisfies Assumption \ref{ass:f}, Lemma \ref{lem:CE-Q} implies that
\[
\EE^{\QQ}\Big[\int_0^T f_u(p_u(\QQ), q_u(\QQ) ) du\Big] \leq \frac{\Delta}{2} H(\QQ | \PP) < \infty, \quad \QQ \in \sQ.
\]
Therefore, and recalling the conventions regarding expectations in \S \ref{subsec:entropy}, one may define a mapping $U : \mathbb{L}^0 \mapsto [- \infty, \infty]$ via
\begin{equation}\label{eq: UE}
U(G) = \inf_{\QQ\in\sQ} \EE^{\QQ}\Big[G + \int_0^T f_u(p_u(\QQ), q_u(\QQ)) du\Big],
\end{equation}
The thus-defined functional $U$ is called a {\define{monetary utility function}}, and $f$ the \define{penalty function} associated to it. Using \eqref{eq: exp CE} and \eqref{equ:sand}, we obtain entropic
upper and lower bounds for $U$, namely
\begin{equation}\label{equ: U-entr-ineq}
-\delta \log \EE^{\PP}[e^{-G/\delta}] \leq U(G) \leq -\Delta \log
\EE^{\PP}[e^{-G/\Delta}], \quad G \in \el^0.
\end{equation}
In particular, $U(G) < \infty$ holds for all $G \in \el^0$.

It follows in a straightforward way from the above that the following properties are
valid, where $G \in \el^0$, $G' \in \el^0$, and $\seq{G}$ is a nonincreasing sequence in $\el^0$:
 \begin{itemize}
	  \item \textit{Positivity:} $U(0) =0$, and $U(G)\leq U(G')$, for $G \leq G'$,
	  \item \textit{Concavity:} $U(\alpha G + (1-\alpha) G') \geq \alpha U(G) +
  (1-\alpha) U(G')$ for all $\alpha \in [0,1]$.
  \item \textit{Monetary invariance:} $U(G+ a) = U(G) + a$, for all $a\in \mathbb{R}$.
 \item \textit{Fatou property:} Whenever $G_n \downarrow G \in \el^0$ and $\sup_{\mathbb{Q}\in \sQ}\EE^{\QQ}[G_1]<\infty$,  we have
  \[ U(G) = \, \downarrow\lim_{n\rightarrow \infty}
  U(G_n).\]
 \end{itemize}

\begin{exa} 
The simplest---but far from the only---example of a monetary utility as described above is when the penalty function $f$ satisfies
\[
f(\omega, t, p,q ) = \frac{\eta (\omega, t)}{2}(p^2 + q^2), \quad (\omega, t, p,q ) \in \Omega \times [0, T] \times \R^2,
\]
where $\eta$ is a predictable process such that $\delta \leq \eta \leq \Delta$ holds for constants $0 < \delta \leq \Delta < \infty$. Here, $\eta(\omega, t)$ may be loosely interpreted as a state-time dependent, on $(\omega, t) \in \Omega \times [0, T]$, risk tolerance coefficient. For constant $\eta$, Lemma \ref{lem:CE-Q} implies that $U$ is entropic utility.
\end{exa}

\section{Single-Agent Maximization Problem} \label{sec: maximization}

\subsection{The financial market} \label{subsec:market}
Our model of
a financial market features one liquidly traded \define{risky asset}, whose
value, denoted in terms of a prespecified num\'{e}raire which we normalize
to $1$,  is given by
\begin{equation}\label{eq: B lam}
 d\Bl_t = \lambda_t \,dt + dB_t, \quad t\in [0,T],
\end{equation}
for some $\lambda \in \bmo$.
Given that it will play a role of a ``free
parameter'' in our analysis, the volatility in \eqref{eq: B lam} is
normalized to $1$; this way, $\lambda$ can simultaneously be interpreted as
the \define{market price of risk}. For technical reasons explained below,
it will be enough to assume for our purposes that
$\lambda \in \bmo$.   The reader should consult the subsection
`The ``fast-and-slow'' model' in the introduction for the proper economic
interpretation of this asset as a concatenation of a continuum of
infinitesimally-short-lived securities.

\subsection{The entropic BMO space}

In order to describe the appropriate regularity class for the agents'
random endowments, which will be larger than $\linf$, we shall need the following space, described via solvability of a certain quadratic BSDE:
\begin{defn}[Entropic BMO]
\label{def:1F98}
A random variable $G \in\lzer$ is said to belong to the
\define{entropic BMO space} $\EBMO$
if there exist (necessarily unique) processes
$(m^{G},n^{G})\in\bmo^2$ and a constant $X^G_0$ such that $ X^G_T = G$ where
 \begin{equation}
 \label{equ:EBMO-BSDE}
  X^G_t = X_0^G + \int_0^t m^G_u\ dB_u + \int_0^t n^G_u\,
    dW_u +  \frac{1}{2} \int_0^t
   \Big((m^G_u)^2+(n^G_u)^2\Big)\, du.
   \end{equation}
\end{defn}


\noindent An exponentiation of the negative of both sides of \eqref{equ:EBMO-BSDE} yields
\begin{align}
\label{equ:def-ME}
\EN(-M^G)_T = e^{-G} \quad \ewhere M^G= X^G_0 + m^G\cdot B + n^G\cdot W \in \BMO,
\end{align}
meaning that $G \in \EBMO$ if and only if $e^{-G}$ is the last element of a
stochastic exponential of a $\BMO$ martingale. Less formally, $\EBMO = -
\log \EN( \BMO)$. Characterization and properties of $\EBMO$ are presented separately in Appendix \ref{app: EBMO}.

For $G \in \EBMO$ we define the following seminorm-like quantity which, in
an abuse of terminology, we still call an \define{EBMO semi-norm}: 
\[ \norm{G}_{\EBMO} := \norm{M^G}_{\BMO}=\norm{(m^G,n^G)}_{\bmo^2}.\]
Since $\norm{\cdot}_{\EBMO}$ lacks the homogeneity property, 
we also introduce the following family:
\[ \norm{G}_{\EBMO,\delta} = \delta \norm{G/\delta}_{\EBMO}, \text{ for }
\delta>0,\]
and note that $G/\delta \in \EBMO$ if and only if the 
equation
 \begin{equation}
 \label{equ:EBMO-BSDE-del}
    X^{G,\delta}_t = X_0^{G,\delta}+ \int_0^t
    m^{G,\delta}_u\ dB_u + \int_0^t n^{G,\delta}_u\,
    dW_u + \tfrac{1}{2 \delta} \int_0^t
   \Big((m^{G,\delta}_u)^2+(n^{G,\delta}_u)^2\Big)\, du,
   \end{equation}
   with $X^{G,\delta}_T=G$,
   admits a (necessarily unique) solution with $(n^{G,\delta},
   m^{G,\delta})\in\bmo^2$. In that case we necessarily have
$X^{G,\delta} = 
\delta X^{G/\delta}$ and $(m^{G,\delta},n^{G,\delta})= \delta
(m^{G/\delta},n^{G/\delta})$, 
so that
   \[ \norm{G}_{\EBMO,\delta} =
   \norm{(m^{G,\delta},n^{G,\delta})}_{\bmo}.\]

\subsection{Agent's utility-maximization problem}

In the market model of \S \ref{subsec:market}, we consider a single \define{economic agent} who trades the risky asset as well as the aforementioned riskless, num\' eraire, asset of constant value 1.  The agent's preferences are modeled by a monetary utility associated to a
penalty function $f$ satisfying Assumption \ref{ass:f}. This agent receives a \define{random endowment} $E \in \el^0$ at time $T$; we shall assume throughout that $E_+ \in \bigcap_{p > 1} \el^p (\PP)$, and $E/\delta  \in \EBMO$.

The agent maximizes the expected utility at the terminal time $T$ arising from trading and random endowment:
\begin{equation}\label{eq: op}
 U(\pi \cdot \Bl_T + E) \rightarrow \max,
\end{equation}
where the portfolio process $\prf{\pi_t}$ represents the number
of shares of the asset kept by the agent, and belongs to an admissible class
described below. As usual, this strategy is financed by investing in or borrowing from the interest-free num\' eraire asset, as needed. To our best knowledge, solution to \eqref{eq: op} for dynamic monetary utility $U$ was missing from the literature. Proposition \ref{prop: verification} below establishes the existence and uniqueness of the optimal portfolio process $\prf{\pi^\lambda_t}$.

For $\ld\in\bmo$, we denote by $\mathcal{M}^\lambda$ the subset of $\sQ$ that consists of equivalent local martingale measures for $B^\lambda$. More precisely, and in view of Levy's characterization theorem, $\mathcal{M}^\lambda$ consists of all probability measures in $\sQ$ under which $B^\lambda$ becomes a Brownian motion.
We note that, since $\lambda\in \bmo$, reverse H\" older inequalities
hold for $\EN(-\ld\cdot B)$ (cf. \cite[Theorem 3.4.]{Kazamaki}) and,
consequently,  the
minimal martingale measure
$\QQ^{\ld}$,  given by $ d\QQ^\lambda / d\PP = \EN(-\ld\cdot B)_T$,
belongs to $\sM^{\ld}$. Note also that any $\QQ \in \mathcal{M}^\lambda$ is such that $d\QQ^\lambda / d\PP = \EN(-\ld\cdot B - q \cdot W)_T$, for appropriate $q \equiv q(\QQ) \in \mathcal{P}^2$.

A strategy $\pi$ is said to be \define{$\ld$-admissible} if $\pi\in \sAl$,
where 
 \[  \mathcal{A}^\lambda = \left\{\pi \in
\mathcal{P}^2\, |\, \pi \cdot B^\lambda \text{ is a } \QQ\text{-supermartingale
for all } \QQ\in \mathcal{M}^\lambda\right\}. \]
 For each $\pi\in\bmo$ and $\QQ\in \mathcal{M}^\lambda$, $\pi\cdot B^{\ld}$ is a $\QQ$-martingale\footnote{For each $\pi\in \bmo$, energy inequalities in \cite[Page 26]{Kazamaki} imply that $\pi \in H^p$ for every $p\geq 1$. This fact combined with $\pare{d\QQ / d\PP} \in \bigcup_{p>1} \mathbb{L}^p$ and H\"{o}lder's inequality imply that $\pi\in H^2(\QQ)$, for each $\QQ\in \mathcal{M}^\lambda$. Therefore $\pi \cdot B^\lambda$ is a $\QQ$-martingale.}; therefore, $\bmo\subseteq \sA^{\ld}$.

The maximization problem in 
\eqref{eq: op}, posed over $\pi\in \sA^{\ld}$, is called the \define{primal
problem}.  The definitions of $U$ and $\sA^{\ld}$ yields the following weak-duality bound
\begin{equation}\label{eq: dual ineq}
 \sup_{\pi \in \mathcal{A}^\lambda} U(\pi\cdot B^\lambda_T + E) \leq
 \inf_{\QQ\in\sM^{\ld}}  \EE^{\QQ}\Big[E + \int_0^T f_u(\lambda_u, q_u (\QQ))
 du\Big],
\end{equation}
with the minimization problem on the right-hand side is called the
\define{dual problem}. 
We remark that the expectation in the definition of the dual
problem exists in $(-\infty,\infty]$, thanks to Proposition
\ref{prop:ebmo} item (2) 
and the $\lpee$-integrability requirement in the
definition of $\sM^{\ld}$.

Our next result characterizes the value of the dual problem via a BSDE.
Given the market price of risk $\lambda \in \bmo$, $f$ satisfying Assumption \ref{ass:f}, and a random endowment $E \in \el^0$ such that $E_+ \in \bigcap_{p > 1} \el^p (\PP)$ and $E/\delta  \in \EBMO$, define the process
  \begin{align}
  \label{equ:def-Yl}
  \Yl_t = \essinf\Big\{\EE^{\QQ}_t \Big[E + \int_t^T f_u(\lambda_u, q_u(\QQ)) du\Big] \, \Big| \, \QQ\in \mathcal{M}^\lambda \Big\}, \quad t\in[0,T].
  \end{align}
Before characterizing $\Yl$, we introduce the partial conjugate $h : \Omega \times [0, T] \times \R^2 \mapsto \R$ in the second spatial argument of $f$:
\begin{equation} \label{equ:h-def}
h_{\omega, t}(p,\nu) = \sup_{q\in\R} \Big( q\nu - f_{\omega, t} (p,q) \Big),\quad (\omega, t) \in \Omega \times [0, T], \ (p,\nu)\in\R^2,
\end{equation}
  and gather some of its properties in the following Lemma:
  \begin{lem}
  \label{lem:h}
  Under Assumption \ref{ass:f}, the partial convex conjugate $h$
  of $f$, given by \eqref{equ:h-def} above has the following properties for all $(\omega, t) \in \Omega \times [0,T]$, whose dependence is hidden below, and
  $(p,\nu)\in\R^2$, where all constants depend only on $\delta$ and
  $\Delta$ of Assumption \ref{ass:f}:
  \begin{enumerate}
    \item
$h(\cdot, \cdot)$ is concave in the first argument and convex in the second, and satisfies  
   \[ -\tfrac{\Delta}{2} p^2 + \tfrac{1}{2\Delta} \nu^2 \leq h (p,
   \nu) \leq -\tfrac{\delta}{2}p^2 + \tfrac{1}{2\delta}\nu^2.
   \]
    \item $h (\cdot, \cdot) \in C^2(\R^2)$, $h (0, 0) =0$, $D h (0,0)=(0,0)$ and there exist
    constants $\gamma>0$ and $\Gamma > 0$, 
    such that
    \[ \partial_{11} h (p,\nu) \leq - \gamma \quad \eand \quad \abs{\partial_{jk} h (p,\nu)}
    \leq \Gamma, \quad\text{ 
    for  $j,k\in {1,2}$.}\]
    \item There exists a constant $\Theta>0$ such that, for all $p,\tp,
    \nu, \tilde{\nu}\in\R$, we have
 \begin{equation}
 \label{equ:der-h-bds}
    \abs{\partial_1 h (p,\nu)} + \abs{\partial_2 h (p,\nu)} \leq \Theta
   \big( \abs{p}+\abs{\nu} \big),
 \end{equation}
and
     \begin{equation}
     \label{equ:h-lip}
        \abs{h_t(p,\nu) - h_t(\tilde{p},\tilde{\nu})} \leq \Theta \big(
       \abs{p} \vee \abs{\tilde{p}} + \abs{\nu}\vee \abs{\tilde{\nu}} \big) \Big(
       \abs{p-\tilde{p}} + \abs{\nu-\tilde{\nu}} \Big).
     \end{equation}
  \item With $\gamma$ as in (2) above, we have 
  \[ h (p,\nu) - p \partial_1 h (p,\nu) \geq \tot \gamma p^2.\]
  \end{enumerate}
  \end{lem}

Now we characterize $\Yl$ via a BSDE in the following result:

\begin{prop}\label{prop: Y BSDE}
Let $\ld\in\bmo$, $f$ satisfying Assumption \ref{ass:f}, and $E \in \el^0$ be such that $E_+ \in \bigcap_{p > 1} \el^p (\PP)$ and $E/\delta  \in \EBMO$. Then, the process $\Yl$ admits a continuous modification and has the following properties:
\begin{enumerate}
  \item[(i)] 
  $X^{E,\delta}_t \leq \Yl_t \leq  \EE^{\QQ^\lambda}_t[E_+] + \tfrac12 \Delta\|\lambda\|^2_{\bmo(\QQ^\lambda)} < \infty$, for all
  $t\in [0,T]$, where $X^{E,\delta}$ is given as in \eqref{equ:EBMO-BSDE-del}; 
  \item[(ii)] $\Yl$ is the unique solution to the BSDE
  \begin{equation}\label{eq: BSDE Y}
   dY_t = \big(h_t(\lambda_t, \nu_t) + \lambda_t \mu_t\big)\, dt + \mu_t\,
   dB_t + \nu_t\, dW_t, \quad Y_T = E,
  \end{equation}
  with $(\mu, \nu)\in \bmo^2$, where $h$ is given by \eqref{equ:h-def}. 
 \end{enumerate}
\end{prop}

Our next task is to identify the optimal investment strategy for the primal problem, using the solution of the dual.

\begin{prop}\label{prop: verification}
Let $\ld\in\bmo$, $f$ satisfying Assumption \ref{ass:f}, and $E \in \mathbb{L}^0$ be such that $E_+ \in \bigcap_{p > 1} \el^p (\PP)$ and $E/\delta  \in \EBMO$. Furthermore, let $(\mul, \nul)$ be the processes featuring  in the martingale component of the (unique) solution $\Yl$ to \eqref{eq: BSDE Y}. Then, the process
\begin{equation}\label{eq: op pi}
\pil = - \partial_1 h(\lambda, \nul) - \mul.
\end{equation}
belongs to $\bmo$ and is the unique optimal investment strategy for the primal problem \eqref{eq: op}.
\end{prop}

\section{Equilibrium}

\subsection{Equilibrium}
We consider a finite number $I\in \mathbb{N}$ of economic agents.
Their preferences are modelled by monetary utilities with penalty functions $(f^i)_i$, and receive random endowments $(E^i)_i$. We impose the following
assumption.

\begin{ass} \label{ass:eq}
For each $i$, $f^i$ satisfies Assumption \ref{ass:f}, with constants $\delta_i \leq \Delta_i$, and $E^i \in \el^0$ is such that $E^i_+ \in \bigcap_{p > 1} \el^p (\PP)$ and $E^i/{\delta_i}  \in \EBMO$.
\end{ass}
In the context of Assumption \ref{ass:eq}, we set
\begin{equation} \label{equ:dD}
   \delta = \min_i \delta_i, \qquad \Delta = \max_i \Delta_i,
\end{equation}
and introduce the shortcuts
$X^i = X^{E^i,\delta_i}$ and $(m^i, n^i)=(m^{E^i,\delta_i},
n^{E^i,\delta_i})\in \bmo$, so that 
\begin{equation}\label{equ:Xis}
 dX^i_t = m^i_t dB_t + n^i_t dW_t + \tfrac{1}{2\delta_i} \big((m^i_t)^2 + (n^i_t)^2\big) dt, \quad X^i_T = E^i.
 \end{equation}

The pair $(\aE,\af)$, where $\aE=(\Ei)_i$, $\af=(f^i)_i$, of
endowments and penalty functions fully characterizes the behavior
of the agents in the model; we call it the \define{population
characteristics}---$\aE$ is the \define{initial allocation} and $\af$
the \define{risk profile}.  
Given a market price of risk process $\ld$, each agent maximizes the expected utility of trading and random endowment 
in the incomplete financial market
of \eqref{eq: op}.

\begin{defn}[Equilibrium] \label{def:equilibrium}
For a population with
characteristics $(\aE,\af)$, a process $\ld\in\bmo$ is called an \define{equilibrium (market price of
risk)} if there exists an $I$-tuple $(\pii)_i$ such that
 \begin{enumerate}
  \item[i)] each $\pii$ is an \emph{optimal strategy} for the agent
  $i$ under $\ld$, i.e.
      \[
       \pii \in  \text{argmax}_{\pi \in \sAl}
       \ee{U^i(\pi\cdot \Bl_T+\Ei)},
      \]
  \item[ii)] the \emph{market clears}, i.e.,
   $\sum_{i} \pii=0$.
 \end{enumerate}
The set of all  equilibria is denoted by $\Ld(\aE,\af)$.
\end{defn}

\begin{rem}
While it is conceivable that an equilibrium market price of risk $\ld$ may exist
outside $\bmo$, we restrict our attention only to the latter class. It is a natural
ambient space, given our assumptions on $\aE$. Moreover, when $\ld\not\in\bmo$ 
there are no known workable
conditions which guarantee the existence of optimal strategies for our 
agents. Therefore, we include the condition $\ld\in\bmo$
in the very definition of an equilibrium, and make all our uniqueness
statements with respect to this class, only. 
\end{rem}

\subsection{A BSDE characterization of equilibria}
The BSDE-based description in Propositions \ref{prop: Y BSDE} and
\ref{prop: verification} of the solution of a single agent's optimization
problem is the main ingredient in the following characterization.
 
\begin{thm}[BSDE characterization of equilibria]\label{thm:char}
Given $\ld\in\bmo$, and population characteristics $(\aE,\af)$ which
satisfy Assumption \ref{ass:eq}, 
the following are
equivalent:
\begin{enumerate}
\item $\ld\in\Ld(\aE,\af)$, i.e.,
$\ld$ is an equilibrium for the population $(\aE,\af)$.
\item
$\ld$ and some processes $(\Yil,  \mu^i, \nu^i)_i$, with each $(\mu^i, \nu^i) \in \bmo^2$, satisfy the following BSDE system:
\begin{equation}
\label{equ:system}
\left\{
\begin{split}
 &d\Yil_t = \big(h^i_t(\lambda_t, \nu^i_t) + \lambda_t \mu^i_t\big) dt + \mu^i_t dB_t + \nu^i_t dW_t, \quad \Yil_T = E^i,\, i=1, \dots, I,\\
 & \textstyle \sum_i \partial_1 h^i(\lambda, \nu^i) = - \sum_i \mu^i.
\end{split}\right.
\end{equation}
\end{enumerate}
\end{thm}

\begin{rem}\ 
\begin{enumerate}
  \item Given the results of Lemma \ref{lem:h}, under the conditions imposed on the drivers $f^i$, the system in \eqref{equ:system} is a genuine system of
  BSDE. Indeed, under Assumption \ref{ass:eq},  each
  $h^i$ is a strictly concave in the first variable, 
  for each value of the second variable. This way, 
    the condition $\partial_1 h^i(\lambda, \nu^i) = -  \sum_i \mu^i$ 
    can be rewritten as 
  \[ \ld = H^{-1} \left( -\oo{I}  \sum_i \mu^i ; \, (\nu^i)_i \right) \ewhere 
 H(p; \,  (\nu_i)_i) = \tfrac{1}{I} \sum_{i=1}^I \partial_1 h^i(p,
 \nu^i),
 \]
where $H^{-1}$ denotes inverse in the first spatial argument. This expression for $\ld$ substituted into the first $I$ equations in
 \eqref{equ:system}, yielding a fully coupled system of BSDE with a quadratic
 driver.
\item While quite meaningless from the competitive point of view, the case $I= 1$
in the above characterization still admits a meaningful interpretation. The
notion of an equilibrium here corresponds to the choice of $\ld$ under which
an agent, with random endowment $E\in\EBMO$ would
choose not to invest in the market at all. The system \eqref{equ:system} reduces to a
single equation \[ dY_t = \mu_t dB_t + \nu_t dW_t + g_t(\mu_t, \nu_t) dt, \quad Y_T=E,
\]
where 
\[
 g_{\omega, t}(\mu, \nu) = \sup_{p\in \R} \big(\mu p + h_{\omega, t}(p,\nu)\big), \quad (\omega, t, \mu, \nu) \in \Omega\times [0,T]\times \R^2,
\]
is the convex conjugate of $f$ and satisfies
\[
 \tfrac{1}{2\Delta} (\mu^2 + \nu^2) \leq g_{\omega, t}(\mu, \nu) \leq \tfrac{1}{2\delta} (\mu^2 + \nu^2), \quad (\omega, t, \mu, \nu) \in \Omega\times [0,T]\times \R^2.
\]
Since $E/\delta\in \EBMO$, Proposition \ref{prop:ebmo} item (1) implies $E/\Delta\in \EBMO$ as well. 
Therefore the previous BSDE admits a unique solution,
highlighting the role of $\EBMO$ as the natural space in the context of
stochastic equilibria with monetary utilities.
\end{enumerate}
\end{rem}

\subsection{Existence and uniqueness}
Now follows our main result.

\begin{thm}[Existence and uniqueness of equilibrium]
\label{thm: main}
 Suppose that the population characteristics $(\aE,\af)$ satisfy Assumption
 \ref{ass:eq}.
For $\delta$ and $\Delta$ given by \eqref{equ:dD},   there exists a
 constant $M = M(\delta,\Delta)>0$ 
 such that whenever
 \begin{equation}\label{equ:mn-bmo}
  \norm{E^i}_{\EBMO,\delta_i} \leq M, \quad \text{ for each $i$,}
 \end{equation}
 there exists a unique equilibrium $\lambda\in \bmo$. 
 Moreover, the triplet $(\aY, \amu, \anu)$, whose components are 
 defined in Proposition \ref{prop: Y BSDE}, is the unique solution 
 to \eqref{equ:system} with $(\amu, \anu)\in \bmo^{2I}$.
\end{thm}
\begin{rem}\
\begin{enumerate}
\item[(1)] The requirement $\norm{E^i}_{\EBMO,\delta_i}\leq M$ can be
fulfilled in several ways. The most important ones are:
\begin{enumerate}
\item By Proposition \ref{prop:ebmo}, $\norm{E^i}_{\EBMO,\delta_i} = \delta_i \norm{E^i / \delta_i}_{\EBMO} \leq 2 \sqrt{\delta_i \norm{E^i}_{\linf}}$. Therefore, ``smallness'' in EBMO is implied by ``smallness'' in $\linf$ of the random endowment $E^i$.
\item By Proposition \ref{prop:suf-ebmo}, when $E^i$ is Malliavin differentiable with bounded Malliavin derivatives (i.e., Malliavin-Lipschitz in the terminology of Appendix \ref{app: EBMO}), its EBMO norm is controlled by $L \sqrt{T}$, where
$L$ is the Malliavin-Lipschitz constant of $E^i$ and $T$ is the time-horizon.
Therefore, our result guarantees the existence of equilibria even when $E^i$ are
unbounded if either the time-horizon or their Malliavin-Lipschitz constants
are small enough. A similar ``smallness in time-horizon"
result has been proven in \cite[Theorem 3.1]{ChoLar14} (and in \cite{Zit06} in
a simpler model) in a
Markovian setting.
\end{enumerate}
 \item[(2)] The constant $M$ in condition \eqref{equ:mn-bmo}, does not depend on the number of
 agents $I$. This is in contrast to ``smallness"-type result of Tevzadze
 (see \cite[Proposition 1]{Tevzadze}) whose condition depends on the number
 of equations in the system. This feature will be important in Corollary
 \ref{cor: homo} later.
 \item[(3)] The uniqueness statement in Theorem \ref{thm: main} is a global
 one, in contrast to the usual local uniqueness in a ball of $\bmo$ which
 follows directly from Banach's fixed point theorem, see e.g. \cite[Proposition 1]{Tevzadze}. A similar
 global uniqueness has been obtained in
\cite[Theorem 4.1]{Kramkov-Pulido} for a different quadratic BSDE
system arising from a price impact model.
\end{enumerate}
\end{rem}

\begin{rem}[Exponential utilities]
When Theorem \ref{thm: main} is specialized to the case of entropic utilities with heterogeneous risk tolerances $(\delta^i)_i \in (0, \infty)^I$, i.e., when
\[
f^i(p,q) = \frac{\delta^i}{2}(p^2 + q^2), \quad i=1, \dots, I,
\]
two additional remarks can be made:
\begin{enumerate}
 \item[(1)] A collection of feasible allocation $\aE$ is Pareto optimal if and
 only if all $E^i/\delta^i$ agree up to constants, i.e., there exists
 $E^c\in \mathbb{L}^0$ and constants $(c^i)_i$ such that
 $E^i/\delta^i = E^c +c^i$ for all $i$. When $E^c_+ \in \cap_{p>1} \mathbb{L}^p(\PP)$ and $E^c \in \EBMO$, statement of
 Theorem \ref{thm: main} still holds when condition \eqref{equ:mn-bmo} is
 translated closer to some Pareto optimal allocation, i.e., 
 \[
 \max_i \|(m^i-  m^c, n^i- n^c)\|_{\bmo(\PP^c)}\leq r,
 \]
 where $\tfrac{d\PP^c}{d\PP} = \tfrac{\exp(-E^c)}{\EE[\exp(-E^c)]}=\mathcal{E}(-\int m^c_u dB_u - \int n^c_u dW_u)_T$
 and $(m^i,n^i)$ are as in \eqref{equ:Xis}.
 \item[(2)] In a Markovian setting where $\aE= \ag(X_T)$ for bounded and
 H\"{o}lder continuous $\ag$, and a diffusion $X$ driven by $B$ and $W$,
 \cite[Theorem 3.1]{Xing-Zitkovic} proves the global existence and
 uniqueness of equilibrium. This result is obtained using an  analytic
 approach, and is only applicable in the Markovian setting.
\end{enumerate}
\end{rem}

\subsection{An economic implication of Theorem \ref{thm: main}}

A novel and interesting feature of \eqref {equ:mn-bmo} is its lack of
dependence on the number of agents $I$; this has profound economic effects
and leads to the existence of
equilibria in an economically meaningful
asymptotic regime with ``large" number of agents.
Given a \define{total endowment} $E_\Sigma\in \mathbb{L}^\infty$ to be shared
among $I$ agents, i.e., $\sum_i E^i = E_\Sigma$, one can ask the following
question: how many and what kind of agents need to share this total endowment
so that they can form a financial market in which an equilibrium exists?
The answer
turns out to be ``sufficiently many
sufficiently homogeneous agents''. In order show that, we first make precise
what
we mean by sufficiently homogeneous. For the population characteristics
$\aE=(\Ei)_i$  and  $\af=(f^i)_i$, with $\aE\in (\linf)^I$ and $\af$
satisfying Assumption \ref{ass:eq}, we define the
\define{endowment heterogeneity index} $\chi^E(\aE) \in [0, 1]$
by
\[
\chi^E(\aE) = \max_{i,j} \frac{\norm{E^i-E^j}_{\linf}}{\norm{E^i}_{\linf} + \norm{E^j}_{\linf}}.
\]
We think of a population of agents as ``sufficiently homogeneous'' if $\chi^E
(\aE)\leq \chi^E_0$ for some, given, critical index $\chi^E_0$. With this in
mind, we have the following corollary of Theorem \ref{thm: main}:

\begin{cor}[Existence of equilibria for sufficiently many sufficiently
homogeneous agents]\label{cor: homo} Given a critical endowment homogeneity
index $\chi^{E}_0 \in [0,\tot)$ and total endowment
$E_{\Sigma}\in\linf$, there exists a constant $I_0 = I_0(
\norm{E_{\Sigma}}_{\linf}, \chi^E_0, \delta, \Delta)\in\N$, so that any
population $(\aE,\af)=(\Ei,f^i)_i$ satisfying Assumption \ref{ass:eq} and \[ I\geq I_0,\quad
\textstyle\sum_i E^i = E_{\Sigma},\quad \text{ and } \quad
\chi^{E}(\aE)\leq \chi^{E}_0,\] admits a unique equilibrium.  \end{cor}

\section{Proofs} \label{sec: proof}
\subsection{Proof of Lemma \ref{lem:CE-Q}}

For the first identity, given $\QQ\in\sQ$, let $Z$ be a continuous
version of the martingale $Z_t = \EE^\PP_t[\tfrac{d\QQ}{d\PP}]$. The $\lpee$-integrability of $Z$
for small enough $p>1$ 
and the convexity of $\vp(z)=z\log z$ imply that $\vp(Z)$ is a uniformly $\PP$-integrable 
submartingale, and, therefore, of class (D) on $[0,T]$.
The semimartingale decomposition 
\[
d\varphi(Z_t) = \tfrac12 Z_t (p_t^2 + q_t^2) dt + \varphi'(Z_t) Z_t p_t
dB_t + \varphi'(Z_t) Z_t q_t dW_t,
\]
where $p \equiv p(\QQ)$ and $q \equiv q(\QQ)$, and a localization argument based on the class (D) property give
\[
H(\QQ|\PP)= \EE^\PP[\varphi(Z_T)] = \tfrac12 \EE^{\PP} \big[\int_0^T Z_t
(p_t^2 + q_t^2) dt\big] = \tfrac12 \EE^{\QQ}\big[\int_0^T (p_t^2 + q_t^2)
dt\big],
\]
where the last equality follows from integration-by-parts and another localization argument.

We now move to the proof of \eqref{eq: exp CE}. We shall prove the special case $\delta = 1$, since the general case follows by simply applying the special case to $G / \delta$. First, assume that $G$ is bounded from below, i.e., $G_- \in \el^\infty$. A use of Jensen's inequality applied to the exponential function yields
 \[
  - \log \EE^\PP[e^{-G}] = - \log
  \EE^{\QQ}\big[\exp\big(- \big(G+ \log
  \tfrac{d\QQ}{d\PP}\big)\big)\big] \leq \EE^{\QQ} \big[G + \log
  \tfrac{d\QQ}{d\PP}\big],
 \]
for all $\QQ\sim\PP$. Furthermore, for $\QQ^G \sim \PP$ satisfying $ d\QQ^G / d\PP =
\tfrac{\exp(-G)}{\EE^\PP[\exp(-G)]}$, which is well defined and an element of $\mathcal{Q}$ because $G_- \in \el^\infty$, we have the equality  $- \log \EE^\PP[e^{-G}] =  \EE^{\QQ^G}  [G + \log
(\QQ^G / d\PP) ]$. Therefore, \eqref{eq: exp CE} follows whenever $G_- \in \el^\infty$. 

For general $G \in \el^0$, it holds that $- \log \EE^\PP[e^{- \max \{G,-n\}}] =  \inf_{\QQ \in \mathcal{Q}} \EE^{\QQ}  [\max \{G, -n\} + \log
( d\QQ / d\PP) ]$ for all $n$ from what we have just proved. Taking the infimum over $n$ in both sides of the last equality, and using the monotone convergence theorem on the left-hand-side, and interchanging the two infima and using the convention regarding expectation from \S \ref{subsec:entropy}, \eqref{eq: exp CE} follows.

\subsection{Proof of Lemma \ref{lem:h}}
 We suppress the subscript $t$ throughout the proof.
  Statement (1) follows by direct inspection and \eqref{equ:sand}. For (2),
  we start by noting that Assumption \ref{ass:f} provides additional bounds
  for second-order partial derivatives of $f$. Indeed, the constants $\delta$ and
    $\Delta$ have the property that 
    \[ \delta \leq \tot(\partial_{11} f + \partial_{22} f) \leq \Delta \quad
    \eand \quad
    \delta^2 \leq \partial_{11} f \partial_{22} f - \partial_{12}^2 f \leq
    \Delta^2 ,\]
    and, so, with 
    $x= \partial_{11} f\geq 0$ and $y=\partial_{22}
    f\geq 0$, we have
    \[  x + y \leq 2 \Delta \quad \eand \quad
    xy \geq \delta^2.\]
It follows immediately that $\Delta - \sqrt{\Delta^2 -
\delta^2} \leq x, y\leq \Delta + \sqrt{\Delta^2 -
\delta^2}$, so both $\partial_{11} f$ and $\partial_{22} f$
are bounded from above and bounded bounded away from $0$, by positive constants that depend only $\delta$ and
$\Delta$.
Since $\partial^2_{12}f \leq \partial_{11} f \partial_{22} f$, hence the absolute values
of all second-order
partial derivatives of $f$ are bounded.

We can deduce from this, 
by the mean-value theorem and the convexity of $f$ in the second argument,
that $q\mapsto \partial_2 f_t(p, q)$ is continuous and strictly (at least linearly, in
fact) increasing,
and that its range is $\R$, for each value of $p$. Consequently, for each
$(p,\nu)\in\R^2$, the equation
  \[ \nu = \partial_2 f(p,q) \]
  has a unique solution, which we denote by $q(p,\nu)$. The
  implicit-function theorem further implies that $q$ is a $C^1$ function of
  both of its arguments. 
Noting that 
  \[ h(p,\nu) = q(p,\nu) \nu - f(p,q(p,\nu)),\]
we conclude that $h\in C^1$ and,
  upon differentiating both sides in both arguments, obtain
  \[ \partial_1 h(p,\nu) = - \partial_1 f(p, q(p,\nu)) \quad \eand  \quad \partial_2
  h(p,\nu) = q(p,\nu).\]
  These relations upgrade the regularity of $h$ to $C^2$ and allow us to
  perform direct computations which yield
 \[
  \partial_{11}h(p, \nu) = -\frac{\text{det}(D^2 f (p,
  q))}{\partial_{22} f(p, q)}, \,
  \partial_{12} h(p, \nu) =-\frac{\partial_{12} f(p,q)}{\partial_{22} f(p, q)}, \,
  \eand 
  \partial_{22} h(p, \nu) =\frac{1}{\partial_{22} f(p, q)}.\]
 The lower bound on $\partial_{22} f$ obtained above, and the
 original bounds from Assumption \ref{ass:f}, imply (2). The equality $D h(0,0)=(0,0)$ is
 a direct consequence of (1). 

For (3), we use the fact that
all second derivatives of $h$ are uniformly bounded, together with
$D h(0,0)=(0,0)$, to conclude that \eqref{equ:der-h-bds},  
for some constant $\Theta$, 
for all $(p,\nu)$. The Lipschitz property \eqref{equ:h-lip} follows from
  \eqref{equ:der-h-bds} by the mean-value theorem.

 Turning to (4), we use the mean-value theorem again to obtain
 \[ h(p,\nu) - p \partial_1 h(p,\nu)  + \tot p^2 \partial_{11} h(\tilde{p},\nu)
 = h(p,0),\]
 for some $\tilde{p}$. It remains to use the bounds in (2) and
   the fact that $h(p,0)\geq 0$, for all $p$.  
   
\subsection{Dynamic monetary utility and its BSDE representation}
A dynamic version of the monetary utility $U$ in \eqref{eq: UE} can be defined for $G \in \el^0$ via
\begin{align}
\label{equ:dyn-U}
 U_t(G) = \essinf \Big\{\EE^\QQ_t \Big[G + \int_t^T f_u(p_u(\QQ), q_u(\QQ)) du\Big]\,
 \Big|\, \QQ \in \mathcal{Q} \Big\}, \quad t\in [0,T].
 \end{align}
The conditional versions of the bounds in \eqref{equ: U-entr-ineq} are, of course, valid. It is shown in \cite{Del-Peng-Gianin} that all time consistent dynamic monetary utilities are of a similar form.

The following characterization of $U=(U_t)_{t\in [0,T]}$ is obtained in \cite[Theorem 2.2]{DelHuBao10}. We record it here in order to also introduce some notation needed for later. Note that it only involves bounded random variables; we shall use this result in a ``localization'' argument in the proof of Proposition \ref{prop: verification}.

 \begin{lem}\label{lem: U BSDE}
  For any $G \in \mathbb{L}^\infty$, $U$ admits a continuous modification which is the unique solution to 
   \begin{equation}\label{eq: BSDE U}
   dU_t = g_t(\mu_t, \nu_t) dt + \mu_t dB_t + \nu_t dW_t, \quad U_T = G,
  \end{equation}
  with $(\mu, \nu)\in \bmo^2$. Above, $g : \Omega \times [0, T] \times \R^2 \mapsto \R$ defined as 
  \[
  g_{\omega, t}(\mu, \nu) = \sup_{p, q\in \mathbb{R}} \big(\mu p + \nu q -
    f_{\omega, t}(p,q)\big), \quad (\omega, t, \mu, \nu) \in \Omega \times [0,T]\times \R^2,
  \]
  is the convex conjugate 
  of the penalty function $f$, and satisfies
  \[
   \tfrac{1}{2\Delta} (\mu^2 + \nu^2) \leq g_{\omega, t} (\mu, \nu) \leq \tfrac{1}{2\delta} (\mu^2 + \nu^2), \quad (\omega, t, \mu, \nu) \in \Omega \times [0, T] \times \R^2.
  \]
\end{lem}
 
 \begin{rem}
Lemma \ref{lem: U BSDE} follows from \cite[Theorem 2.2, $2 \Rightarrow 6$]{DelHuBao10}, which can be generalized to random penalty function satisfying Assumption \ref{ass:f}. (The penalty function $f$ is assumed to be deterministic in \cite{DelHuBao10}.) Indeed, $2 \Rightarrow 3$ and $4 \Rightarrow 6$ in \cite[Theorem 2.2]{DelHuBao10} hold for random function $f$ satisfying uniform growth condition \eqref{equ:sand}, and $3 \Rightarrow 4$ in \cite[Theorem 2.2]{DelHuBao10} is proved in \cite[Theorem 5.2 $(iv) \Rightarrow (v)$]{Jouini-Schachermayer-Touzi}.
\end{rem}

\begin{rem}
In the notation of Lemma \ref{lem: U BSDE}, the probability measure $\hat{\QQ}$, given by
\[
\frac{d\hat{\QQ}}{d\PP} = \mathcal{E} \left(-\int \partial_1 g_u (\mu_u, \nu_u)\, dB_u
  -\int \partial_2 g_u (\mu_u, \nu_u)\, dW_u \right)_T,
\]
is the unique minimizer in \eqref{equ:dyn-U} above. Since $g$ is convex and of quadratic growth in the spatial arguments, its partial derivatives $\partial_j g$, $j=1, 2$, grow at most linearly.  Given that $(\mu, \nu)\in\bmo$, we have $\partial_j g(\mu, \nu)\in \bmo$, as well, and the fact that $\hat{\QQ} \in\sQ$ follows from the reverse H\"{o}lder inequality (cf.  \cite[Theorem 3.1]{Kazamaki}).
\end{rem}

\subsection{Proof of Proposition \ref{prop: Y BSDE}}
Part (i): First, the assumptions $E_+\in \cap_{p>1} \mathbb{L}^p(\PP)$ and $d\QQ^\lambda/d\PP \in \cup_{p>1} \mathbb{L}^p(\PP)$, combined with H\"{o}lder's inequality, imply that $E_+\in \mathbb{L}^1(\QQ^\lambda)$. 
The bounds in \eqref{equ:sand} and Lemma \ref{lem:CE-Q} below it, together with the assumption $\lambda\in \bmo$, imply
\[ \delta\log \EE[ e^{-E/\delta}] \leq \Yl_0 \leq \|E_+\|_{\mathbb{L}^1(\QQ^\lambda)}+ \tfrac12 \Delta\|\lambda\|^2_{\bmo(\QQ^\lambda)}.
\]
Applied conditionally, the same argument can be used to extended the validity
of the above inequalities for each $t\in [0,T]$. It remains to note that
$\delta \log \EE_t[ e^{-E/\delta}] = \delta X^{E/\delta}_t = X^{E,\delta}_t$. 

\medskip

Part (ii): When $E$ is bounded, the claim that $\Yl$ satisfies \eqref{eq: BSDE Y}
follows from an argument similar to the one in  \cite[Theorem 2.2, $2 \Rightarrow
6$]{DelHuBao10}. When, as assumed,  $E/\delta$ belongs to $\EBMO$, 
the BSDE
characterization \eqref{eq: BSDE Y} is
proved using the localization argument of \cite[Theorem 2]{BriHu06}, thanks
to the bounds for $Y$ in (i). 

The remaining question is whether $(\mu, \nu)\in \bmo$. To show
that it is, in fact, true, 
we first note that 
 \begin{equation}
 \label{equ: X 2}
 \begin{split}
   X^{E,\delta}_\sigma  = E &- 
  \int_\sigma^T \Big(\tfrac1{2\delta}( m^{E,\delta}_u)^2 + \tfrac{1}{2\delta}
  ( n^{E,\delta}_u)^2) -\lambda_u ( m^{E,\delta}_u)\Big)\, du\\  &-\int_\sigma^T
   m^{E,\delta}_u\,  dB^\lambda_u
  -\int_\sigma^T  n^{E,\delta}_u\, dW_u,
\end{split}\end{equation}
  for any stopping time $\sigma$. Thanks to 
 \cite[Theorem 3.6]{Kazamaki}, both $m^{E,\delta}$ and $n^{E,\delta}$ (as well as $\ld$) 
belong to $\bmo(\QQ^{\ld})$, and, so,
 both stochastic integrals on the right-hand side of \eqref{equ: X 2}
 above are $\QQ^\lambda$-martingales. The 
 $\bmo(\QQ^{\ld})$-property of $\ld, m^{E,\delta}$ and $n^{E,\delta}$ allows us to conclude, upon
 a projection of both sides on $\sF_{\sigma}$, 
 that $X^{E,\delta}$ is of class (D) under
 $\QQ^\lambda$. Therefore, 
 the bounds in (i) imply that $\Yl$ is
 of class (D) under $\QQ^\lambda$, as well, so we can use a localization
 argument to conclude that
\[
 \EE^{\QQ^\lambda}_\sigma[E] - \Yl_\sigma = \EE^{\QQ^\lambda}_\sigma
 \Big[\int_\sigma^T h(\lambda_u, \nu_u) du\Big], \text{ for each stopping
 time } \sigma.
\]
Thanks to Lemma \ref{lem:h}, part (1), 
the right-hand side is bounded from below by 
$\EE^{\QQ^\lambda}_\sigma [\int_\sigma^T (-\tfrac{\Delta}{2} \lambda^2_u +
\tfrac{1}{2\Delta} \nu_u^2)du]$, while an upper bound for the left-hand side
is given by
 \begin{align*}
 \EE^{\QQ^\lambda}_\sigma[E] - X^{E,\delta}_\sigma 
 &=  \EE^{\QQ^\lambda}_\sigma \Big[\int_\sigma^T \big(\tfrac1{2\delta}
 (m^{E,\delta}_u)^2 + \tfrac1{2\delta}
 (n^{E,\delta}_u)^2 -\lambda_u m^{E,\delta}_u\big) du\Big]\\ 
 &\leq (\tfrac{1}{2\delta}+\tot) \Big( \|\lambda\|^2_{\bmo(\QQ^\lambda)} + 
 \|m^{E,\delta}\|^2_{\bmo(\QQ^\lambda)} +   \|n^{E,\delta}\|^2_{\bmo(\QQ^\lambda)} \Big).
 \end{align*}
These estimates imply that $\nu\in \bmo(\QQ^{\ld})$, and 
the isomorphism theorem \cite[Theorem 3.6]{Kazamaki} implies that
$\nu\in\bmo$, as well. 

To show that $\mu \in \bmo$, we first prove that $Y-X^{E,\delta}$ belongs to $\mathcal{S}^\infty$. Since
$Y- X^{E,\delta}\geq 0$, it will be enough to show that $Y-X^{E,\delta}$ is bounded from above. 
To this end, we compute the semimartingale decomposition of $Y-
X^{E,\delta}$ under
$\QQ^{\ld}$:
\begin{equation}
\label{eq: Y-X}
\begin{split}
 d(Y_t - X^{E,\delta}_t) &= \big[h(\lambda_t, \nu_t) 
 +  m^{E,\delta}_t \lambda_t - \tfrac{1}{2\delta}  \big((m^{E,\delta}_t)^2 +
 (n^{E,\delta}_t)^2\big) \big] dt\\
 & \quad + (\mu_t -  n^{E,\delta}_t)\, dB^\lambda_t + (\nu_t- 
 n^{E,\delta}_t)\, dW_t. 
\end{split}
\end{equation}
Using the lower bound for $h$, the class (D) property of both $Y$ and
$ X^{E,\delta}$ under $\QQ^\lambda$, and 
the fact that $Y_T-  X^{E,\delta}_T =0$, we obtain
\[
 Y_t -  X^{E,\delta}_t \leq \EE^{\QQ^\lambda}_t \big[\int_t^T
 \tfrac{\Delta}{2} \lambda_u^2 -\tfrac{1}{2\Delta} \nu_u^2 - 
 m^{E,\delta}_u
 \lambda_u + \tfrac{1}{2\delta} (m^{E,\delta}_u)^2 + \tfrac{1}{2\delta}(n^{E,\delta}_u)^2 \, du\big],
\]
where the right-hand side is bounded from above, uniformly in $t$, due to the
$\bmo(\QQ^\lambda)$ property of $\lambda, \nu, m^{E,\delta}$ and
$n^{E,\delta}$.

The obtained bounds in $\bmo$ and $\sS^{\infty}$, used together with
It\^{o}'s formula applied to $(Y- X^{E,\delta})^2$ 
and facilitated by \eqref{eq: Y-X}, imply that
$\mu-m^{E,\delta}\in\bmo(\QQ^{\ld})$. Another appeal 
to \cite[Theorem 3.6]{Kazamaki} yields $\mu\in\bmo$.

\medskip

Finally, we turn to uniqueness  for \eqref{eq: BSDE Y}, and consider two solutions
$(Y, \mu, \nu)$ and $(\tilde{Y}, \tilde{\mu}, \tilde{\nu})$. Their
difference $Y-\tilde{Y}$ satisfies
 \[
  d(Y_t-\tilde{Y}_t) = (h(\lambda_t, \nu_t) - h(\lambda_t, \tilde{\nu}_t)) dt + (\mu_t -\tilde{\mu}_t) dB^\lambda_t + (\nu_t -\tilde{\nu}_t) dW_t.
 \]
 By convexity, we have $h(p, \nu) - h(p, \tilde{\nu}) \leq
 q^*(p,\nu) (\nu
 -\tilde{\nu})$, where $q^*(p,\nu)=
 \partial_2 h(p, \nu)$, so that
\begin{align} \label{equ:Y-Y}
  Y_t - \tilde{Y}_t \geq -\int_t^T (\mu_u - \tilde{\mu}_u) dB^\lambda_u
  -\int_t^T (\nu_u - \tilde{\nu}_u) dW^{q^*}_u,
  \end{align}  
  where $q^*$ denotes the process $q^*(\ld,\nu)$. 
  The bounds in \eqref{equ:der-h-bds} in Lemma \ref{lem:h} imply that
  $q^*\in\bmo$. 
 Therefore, the probability measure $\QQ^{\ld,*}$,
 defined by
 $d\QQ^{\lambda, *}/d\PP = \mathcal{E}(-\int\lambda_u dB_u - \int q^*_u
 dW_u)_T$ is well defined. Moreover, 
 thanks to the
 $\bmo$ property of $(\mu, \nu)$ and $(\tilde{\mu}, \tilde{\nu})$,
 the stochastic integrals on the right hand side of \eqref{equ:Y-Y} above
 are $\QQ^{\lambda, *}$-martingales. 
 A projection onto $\sF_t$ under $\QQ^{\lambda, *}$ of both sides of
 \eqref{equ:Y-Y}
 yields $Y\leq \tilde{Y}$. 
 The reverse inequality is proved similarly.

\subsection{Proof of Proposition \ref{prop: verification}}
 The bounds in \eqref{equ:der-h-bds} in Lemma \ref{lem:h} allow 
 $\pil$ to inherit its 
 $\bmo$
 property from $\lambda$ and $(\mul, \nul)$.
Setting $\bmul = \pil + \mul$ and $\bnul = \nu^\lambda$, we have $\bmul,
\bnul\in\bmo$ and
 \begin{equation}\label{eq: gh}
  h(\lambda, \nul) + \mul \lambda=
  g(\bmul, \bnul) -\pil \lambda,
 \end{equation}
  so that
  \begin{align*}
   d \Yl_t = & \big(g_t(\bmul_t, \bnul_t) -\pil_t \lambda_t\big) dt + (\bmul_t
   - \pi^\lambda_t) dB_t + \bnul_t dW_t\\
   =& g_t(\bmul_t, \bnul_t) dt -\pi^\lambda_t dB^\lambda_t + \bmul_t dB_t + \bnul_t dW_t.
  \end{align*}
  Therefore $\Yl_t + \pil \cdot B^\lambda_t$ satisfies \eqref{eq: BSDE U} with
  the terminal condition $E + \pi\cdot B^\lambda_T$. When $E+\pi^\lambda\cdot
  B^\lambda_T$ happens to be bounded,  uniqueness of 
  \eqref{eq: BSDE U} implies that
 \[
  \Yl_0 = \inf \Big\{\EE^{\QQ} \Big[E + \int_0^T \pi^\lambda_u dB^\lambda_u +
  \int_0^T f_u(p_u, q_u) du\Big]\,\Big|\, \QQ \in \mathcal{Q}\Big\},
 \]
 and the optimality of $\pi$ follows from \eqref{eq: dual ineq}.

  When $E+\pi^\lambda\cdot B^\lambda_T$ is unbounded, we employ a localization
  argument using the nondecreasing sequence $\tau_n = \inf\{t\geq 0
  \,|\, |\Yl_t + \pil \cdot B^\lambda_t|\geq n\}\wedge T$, $n\in \mathbb{N}$
  of stopping times with $\PP[ \tau_n= T]\to 1$. 
  The process $\Yl_t +  \pil\cdot
  B^\lambda_t$ satisfies \eqref{eq: BSDE U} with the bounded terminal condition
  $\Yl_{\tau_n} + \pil \cdot B^\lambda_{\tau_n}$, and, so,  by uniqueness,
  \begin{equation}\label{eq: Yn-rep}
   \Yl_0 = \inf \Big\{\EE^{\QQ} \Big[\Yl_{\tau_n} + \int_0^{\tau_n} \pil_u dB^\lambda_u + \int_0^{\tau_n} f_u(p_u, q_u) du\Big]\,\Big|\, \QQ \in \mathcal{Q}\Big\}.
  \end{equation}
  Therefore, the equality in
  \eqref{eq: Yn-rep} above and the nonnegativity of $f$ yield
   \begin{equation}
   \label{equ: nQ}
   \Yl_0 \leq \EE^{\QQ}\Big[\Yl_{\tau_n} + \int_0^{\tau_n} \pil_u dB^\lambda_u +
   \int_0^T f_u(p_u, q_u) du\Big], \text{ for each } n\in\N \text{ and each }
  \QQ\in \mathcal{Q}.
  \end{equation}
For the first term on the right-hand side, we claim that $\{Y^\lambda_{\tau_n}\}_n$ is bounded from above by a uniformly integrable family under $\QQ$. Indeed, we have from Proposition \ref{prop: Y BSDE} item (i) that $\Yl_t \leq \EE^{\QQ^\lambda}_t[E_+] + \tfrac12 \Delta \|\lambda\|^2_{\bmo(\QQ^\lambda)}$.  On the other hand, $\QQ\in \sQ$ implies that $\tfrac{d\QQ}{d\PP} \in \mathbb{L}^p(\PP)$ for some $p$ sufficiently close to $1$. Moreover, since $\lambda\in \bmo(\QQ^\lambda)$, reverse H\"{o}lder's inequalities (see \cite[Theorem
  3.1]{Kazamaki}) imply that $\tfrac{d\PP}{d\QQ^\lambda} \in \mathbb{L}^{p'}(\QQ^\lambda)$ for some $p'$ sufficiently close to $1$, similarly $\tfrac{d\QQ^\lambda}{d\PP} \in \mathbb{L}^{p^{''}}(\PP)$ for some $p^{''}$ sufficiently close to $1$. Take $s\in (1, p\wedge p'\wedge p^{''})$ and define $q, q'$ and $q^{''}$ via $1/p+1/q=1/p'+1/q' = 1/p^{''} + 1/q^{''} =1$. We have from H\"{o}lder's inequality that 
  \begin{align*}
  & \EE^{\QQ}\Big[\big(\EE^{\QQ^\lambda}_t[E_+]\big)^s\Big] = \EE^{\PP}\Big[\tfrac{d\QQ}{d\PP}\big(\EE^{\QQ^\lambda}_t[E_+]\big)^s\Big]\leq \EE^{\PP}\Big[\big(\tfrac{d\QQ}{d\PP}\big)^p\Big]^{\tfrac{1}{p}} \EE^{\PP}\Big[\big(\EE^{\QQ^\lambda}_t[E_+]\big)^{sq}\Big]^{\tfrac{1}{q}}\\
  & \leq \EE^{\PP}\Big[\big(\tfrac{d\QQ}{d\PP}\big)^p\Big]^{\tfrac{1}{p}} \EE^{\QQ^\lambda} \Big[\tfrac{d\PP}{d\QQ^\lambda} \big(\EE^{\QQ^\lambda}_t[E_+]\big)^{sq}\Big]^{\tfrac{1}{q}} \leq \EE^{\PP}\Big[\big(\tfrac{d\QQ}{d\PP}\big)^p\Big]^{\tfrac{1}{p}} \EE^{\QQ^\lambda} \Big[\big(\tfrac{d\PP}{d\QQ^\lambda}\big)^{p'}\Big]^{\tfrac{1}{qp'}} \EE^{\QQ^\lambda} \Big[E_+^{sq q'}\Big]^{\tfrac{1}{qq'}}\\
  & \leq \EE^{\PP}\Big[\big(\tfrac{d\QQ}{d\PP}\big)^p\Big]^{\tfrac{1}{p}} \EE^{\QQ^\lambda} \Big[\big(\tfrac{d\PP}{d\QQ^\lambda}\big)^{p'}\Big]^{\tfrac{1}{qp'}} \EE^{\PP}\Big[\tfrac{d\QQ^\lambda}{d\PP} E_+^{sq q'} \Big]^{\tfrac{1}{qq'}}\\
  &\leq \EE^{\PP}\Big[\big(\tfrac{d\QQ}{d\PP}\big)^p\Big]^{\tfrac{1}{p}} \EE^{\QQ^\lambda} \Big[\big(\tfrac{d\PP}{d\QQ^\lambda}\big)^{p'}\Big]^{\tfrac{1}{qp'}}  \EE^{\PP}\Big[\big(\tfrac{d\QQ^\lambda}{d\PP}\big)^{p^{''}}\Big]^{\tfrac{1}{qq'p^{''}}} \EE^{\PP} \Big[E_+^{sqq'q^{''}}\Big]^{\tfrac{1}{qq'q^{''}}},
  \end{align*}
  for any $t\in [0,T]$. Therefore (the continuous modification of) the conditional expectation $\EE^{\QQ^\lambda}_\cdot[E_+]$ is a class (D) process under $\QQ$, which confirms the claim that $\{Y^\lambda_{\tau_n}\}_n$ is bounded from above by a uniformly integrable family under $\QQ$. So we can use Fatou's lemma to conclude that
  $\limsup_{n} \EE^{\QQ}[\Yl_{\tau_n}] \leq \EE^{\QQ}[E]$. For the stochastic integral on the right-hand side of \eqref{equ: nQ}, similar argument as above yields
  \begin{equation} \label{equ:Hol-3}
    \EE^{\QQ}\Big[\big| \int_{\tau_n}^T \pil_u dB^\lambda_u\big|\Big] \leq
    \EE^{\PP}\big[\big(\tfrac{d\QQ}{d\PP}\big)^p\big]^{\tfrac1p}
    \EE^{\QQ^\lambda}\big[\big(\tfrac{d\PP}{d\QQ^\lambda}\big)^{p'}\big]^{\tfrac{1}{p'q}}
    \EE^{\QQ^\lambda}\Big[\Big(\int_{\tau_n}^T \pil_u dB^\lambda_u\Big)^{q
    q'}\Big]^{\tfrac{1}{q q'}}, 
  \end{equation}
  where $1/p+1/q=1/p' + 1/q' =1$ and $p, p'$ are sufficiently close to $1$. For the third expectation on the right-hand side, since $\pil \cdot
  B^\lambda \in \BMO(\QQ^\lambda)$, we have $\sup_n \EE^{\QQ^\lambda}
  \big[\big(\int^T_{\tau_n} \pil_u dB^\lambda_u\big)^{2qq'}\big]\leq \| \pil
  \cdot B^\lambda\|^{2qq'}_{\BMO_{2qq'}(\QQ^\lambda)}<\infty$ and 
 the de
  la Vall\'{e}e Poussin theorem implies that $\big(\int^T_{\tau_n} \pil_u
  dB^\lambda_u\big)^{qq'}$ is uniformly integrable in $n$ under
  $\QQ^\lambda$. Thus, the third expectation in \eqref{equ:Hol-3} vanishes as $\tau_n\rightarrow
  T$ and we obtain
  \[
   \Yl_0 \leq \EE^{\QQ} \Big[E + \int_0^T \pil_u dB^\lambda_u + \int_0^T f_u(p_u, q_u) du\Big], \quad \text{ for any } \QQ\in \mathcal{Q}.
  \]
  Therefore,
  \[
    \Yl_0 \leq \inf_{\QQ\in \mathcal{Q}}\EE^{\QQ} \Big[E + \int_0^T \pil_u
    dB^\lambda_u + \int_0^T f_u(p_u, q_u) du\Big] = U\Big(E + \int_0^T \pil_u
    dB^\lambda_u\Big),
  \]
 and the optimality of $\pil$ follows from \eqref{eq: dual ineq}.  Moreover
 the minimal measure is attained at $\hQl\sim\PP$, given by
 $\tfrac{d\hQl}{d\PP} = \mathcal{E}(-\int  \lambda_u dB_u - \int \hql_u
 dW_u)_T$, where  $\hql_t = \partial_2 h(\lambda_t, \nul_t)$. Therefore $\Yl_t +\int_0^t f_u(\lambda_u, \hat{q}^\lambda_u) du$ is a $\hat{\QQ}^\lambda$-martingale.

\medskip

  To prove uniqueness, we take a another optimal strategy $\tilde{\pi}$ and
observe that, thanks to its optimality, the two inequalities in 
 \begin{align*}
  U(E+\tilde{\pi} \cdot B^\lambda_T) &= \inf_{\QQ\in \mathcal{Q}} \EE^{\QQ}
  \big[E+ \tilde{\pi} \cdot B^\lambda_T + \int_0^T f_u(p_u, q_u) \,du\big]\\
  &\leq \EE^{\hQl}\big[E + \tilde{\pi} \cdot B^\lambda_T + \int_0^T
  f_u(\lambda_u, \hql_u) du\big]\\
  &\leq \EE^{\hQl}\big[E + \int_0^T f_u(\lambda_u, \hql_u) du\big] =
  \inf_{\QQ\in \mathcal{M}^\lambda}\big[E+\int_0^T f_u(\lambda_u,
  q_u)du\big],
 \end{align*}
 are, in fact, equalities. 
 In particular, 
 the $\hQl$-supermartingale $\tilde{\pi}\cdot B^\lambda$ is a $\hQl$-martingale, and
 \[
  U(E+\tilde{\pi} \cdot B^\lambda_T) = \EE^{\hQl}\big[E + \tilde{\pi} \cdot
  B^\lambda_T + \int_0^T f_u(\lambda_u, \hql_u)du\big].
 \]
 The previous identity and \cite[Proposition 2.1, item 2)]{DelHuBao10}
 imply that $U_t(E+\tilde{\pi} \cdot B^\lambda_T) +\int_0^t f_u(\lambda_u,
 \hql_u) du$ is a $\hQl$-martingale. $\sF_t$-cash invariance of $U_t$ and the
 $\hQl$-martingale property of $\tilde{\pi} \cdot B^\lambda$ then yield
 that $U_t(E+\int_t^T \tilde{\pi}_u dB^\lambda_u)+ \int_0^t f_u(\lambda_u,
 \hql_u) du$ is a $\hQl$-martingale as well. It is dominated by another
 $\hQl$-martingale, namely, 
 $\Yl_t + \int_0^t f_u(\lambda_u, \hql_u) du$.
 These two martingales, in fact, coincide because they satisfy the same 
 terminal condition. In particular, we have
 \[ Y_t =
 U_t \Big(E + \int_t^T \tilde{\pi}_u dB^\lambda_u\Big),\ t\in [0,T].\]
 It has been shown in \cite[Proposition 2.1, item 1)]{DelHuBao10}
 that $U_t(E+\tilde{\pi} \cdot B^\lambda_T) + \int_0^t f_u(p_u, q_u) du$ is a
 $\QQ$-submartingale, for any $\QQ\in \mathcal{Q}$. 
    This submartingale property combined with
 the $\hQl$-martingale property of $U_t(E+\tilde{\pi} \cdot B^\lambda_T) +
 \int_0^t f_u(\lambda_u, \hql_u) du$ yields
 \[
  U_{t\wedge \tau}(E+ \tilde{\pi}\cdot B^\lambda_T) = \essinf_{\QQ\in \mathcal{Q}} \EE^{\QQ}_{t\wedge \tau} \big[U_\tau(E+ \tilde{\pi}\cdot B^\lambda_T)+ \int_{t\wedge \tau}^{\tau} f_u(p_u, q_u) du\big],
  \]
 for any $[0,T]$-valued stopping time $\tau$.
 In particular, when $\tau=\tau_k$ where $\tau_k = \inf\{t\geq 0: |U_t(E+
 \tilde{\pi} \cdot B^\lambda_T)|\geq k\}\wedge T$, uniqueness for 
 \eqref{eq: BSDE U} with a bounded terminal condition implies that
 \[
  dU_t(E+ \tilde{\pi}\cdot B^\lambda_T) = g_t(\tilde{\mu}_t, \tilde{\nu}_t) dt + \tilde{\mu}_t dB_t + \tilde{\nu_t} dW_t, \quad 0\leq t \leq \tau_k,
 \]
 for some $(\tilde{\mu}, \tilde{\nu})$. Therefore $U_t = U_t(E+\int_t^T \tilde{\pi}_u dB^\lambda_u)$ satisfies
 \[
  dU_t = \big(g_t(\tilde{\mu}_t, \tilde{\nu}_t) - \lambda_t \tilde{\pi}_t\big) dt + (\tilde{\mu}_t -\tilde{\pi}_t) dB_t + \tilde{\nu}_t dW_t.
 \]
 Comparing this with the dynamics  in \eqref{eq: BSDE Y}, and using
 the uniqueness of the
 semimartingale decomposition, we obtain $\mul = \tilde{\mu} - \tilde{\pi}$,
 $\nul=\tilde{\nu} $, and $g_t(\tilde{\mu}, \tilde{\nu})-\lambda \tilde{\pi} =
 h_t(\lambda, \nul) + \lambda \mul$. Therefore
 \[
  \sup_{p\in \mathbb{R}} \big(h_t(p,\nul) + p (\mul +
  \tilde{\pi})\big)=g_t(\mul + \tilde{\pi}, \nul) = h_t(\lambda, \nul) +
  \lambda(\mul + \tilde{\pi}).
 \]
 Concavity of $h_t$ in its first argument yields $\mul + \tilde{\pi} = -
 \partial_1 h_t(\lambda, \nul)$ and confirms $\tilde{\pi} = \pil$.
 
\subsection{Proof of Theorem \ref{thm:char}}
$(1)\Rightarrow (2)$. Given an equilibrium $\lambda\in \Lambda(\aE,\af)$, 
let $\pi^i$ be the primal
optimizer for agent
$i$. The uniqueness statement in Proposition \ref{prop: verification}
identifies 
\[ \pi^i = -\partial_1 h^i(\lambda, \nu^i) -\mu^i,\]
where $(\Yil,
\mu^i, \nu^i)$ is the unique solution of \eqref{eq: BSDE Y} with
terminal condition $\Yil_T = E^i$ and $(\mu^i, \nu^i) \in \bmo$. The
market clearing condition $\sum_i \pi^i = 0$ implies $\sum_i \partial_1
h^i(\lambda, \nu^i) = -\sum_i \mu^i$.

$(2) \Rightarrow (1)$. Given a solution $(\Yil, \mu^i, \nu^i)_i$ to
\eqref{equ:system} with each $(\mu^i,\nu^i)\in \bmo$, we set $\pi^i =
-\partial_1 h^i(\lambda, \nu^i) -\mu^i$. 
Proposition \ref{prop: verification} implies that $\pi^i$ is optimal for
agent $i$ when the market price of risk is $\ld$, 
and the market-clearing condition is satisfied since
$\sum_i \pi^i = - \big( \sum_i
\partial_1 h^i(\lambda, \nu^i) + \sum_i \mu^i\big)=0.$

\subsection{Proof of Theorem \ref{thm: main}} Any object that depends only
on $\delta$ and $\Delta$ from Theorem \ref{thm: main} 
will be called \define{universal}. In particular we will
talk about universal constants $\eps$ and $C$, as well as about a universal
function $\bar{\eps}(M):[0,\eps)\to (0,\infty)$ in the sequel. 
When appearing in the same proof, we allow their values to change from
appearance to appearance, without explicit mention. Moreover, all universal
constants are assumed to be strictly positive. 

We start by setting up a framework for the Banach fixed-point theorem in
the space $\bmo$.  Given $\ld\in\bmo$ and $i \in 1,\dots, I$,  let
$\Yil$ and $(\muil, \nuil)\in\bmo$
be components of the unique solution of 
\[ d\Yil_t = \Big(h^i_t(\ld_t, \nuil_t)+\ld_t \muil\Big)\, dt + \muil_t\, dB_t +
\nuil_t\, dW_t,\qquad \Yil_T= E^i,\]
where 
\[ h^i_{\omega,t}(p,\nu) = \sup_{q\in\R} \Big( q \nu - f^i_{\omega,t}(p,q) \Big), \quad (\omega, t, p,q) \in \Omega \times [0,T]\times \R^2.\]
We fix the random endowments $(E^i)_i$ throughout and remind the reader
that $\Xi$ and $(m^i,n^i)$ are as in \eqref{equ:Xis}.

Let the function $H$ be defined by 
\[
 H_t(p, \anu) = \tfrac{1}{I} \sum_{i=1}^I \partial_1 h^i_t(p, \nu^i), \quad
 \text{ for } t\in [0,T], (p,\anu) \in \R\times\mathbb{R}^I, 
\]
By Lemma \ref{lem:h} items (1) and (2), the function
$p\mapsto \partial_1 h^i_t(p,\nu)$ is strictly decreasing for each $t$, $\nu$ and
$i$,  and its range is $\R$, therefore, $H_t(p, (\nu^i)_i)$ admits an inverse
$H^{-1}_t(\cdot,\anu)$. We use it to define the 
 \define{excess-demand map} $F$ on $\bmo$ by 
 \begin{equation*}\label{eq: demand map}
 F(\lambda_t)= H^{-1}_t(-\tfrac{1}{I} \textstyle\sum_i \muil_t, (\nu^{i, \lambda}_t)_i), \quad t\in [0,T].
 \end{equation*}
The significance of this map lies in the simple fact that $\ld$ is an equilibrium if and only if $F(\ld)=\ld$, i.e., if $\ld$ is a fixed point of $F$.
Our first task is to show that $H^{-1}$
is a Lipschitz function:
\begin{lem}\label{lem: H}
 There exists a universal constant $C$ such that
  \begin{equation}\label{eq:Lip-H}
   |H^{-1}_t(p, \anu) - H^{-1}_t(\tilde{p}, (\tilde{\nu}^i)_i) )| \leq C \Big(
   \abs{p-\tilde{p}}  + \max_i \abs{\nu^i - \tilde{\nu}^i} \Big), 
  \end{equation}
  for all $t\in[0,T]$, $p,\tilde{p}\in\R$, and $\anu, (\tilde{\nu}^i)_i \in \R^I$. 
\end{lem}
\begin{proof}
 The subscript $t$ is suppressed throughout the proof.
 To prove \eqref{eq:Lip-H}, we start from
 \begin{equation}\label{eq:H-est}
 \begin{split}
\abs{H^{-1}(p,\anu) - H^{-1}(\tilde{p},\tanu)} \leq & \quad
\abs{H^{-1}(p,\anu) - H^{-1}(\tilde{p},\anu)}\\ & +
\abs{H^{-1}(\tilde{p},\anu) - H^{-1}(\tilde{p},\tanu)}
\end{split}
\end{equation}
 To estimate the right-hand side, we compute
  \begin{align*} 
  \partial_1 (H^{-1}) ( p , \anu) & = 1/ S(p,\anu), \eand \\
  \partial_{1+j} (H^{-1}) ( p , \anu) &= -
  \oo{I} \partial_{12} h^j(H^{-1}( p , \anu), \nu^j)/ S(p,\anu), \efor
  1\leq j \leq I, 
  \end{align*}
where $S(p,\nu) = \oo{I} \sum_j \partial_{11} h^j(H^{-1} ( p , \anu),
\nu^j)$.
It follows from Lemma \ref{lem:h}, part (2), that 
   \[ |\partial_1 (H^{-1}) (p, \anu)|\leq C \quad \eand \quad |\partial_{1+j} (H^{-1})
   (p, \anu)|\leq C /I,\]
   for each $j$ and all $p\in\R$ and $\anu\in\R^I$. The estimate \eqref{eq:Lip-H} follows by applying the
   mean-value theorem to both terms in \eqref{eq:H-est}.
\end{proof}

Next we present a refinement of the classical result on uniform equivalence of bmo spaces (see
\cite[Theorem 3.6]{Kazamaki}),
based on a result of Chinkvinidze and Mania (see \cite{Chikvinidze-Mania}).
\begin{lem}
\label{lem:eq-bmo}
Let $\sigma\in\bmo$ be such that $\norm{\sigma}_{\bmo} =: \sqrt{2} R$ for some $R<1$. If
\,$\hPP\sim \PP$ is such that $\tfrac{d\hPP}{d\PP} = \EN(\sigma\cdot \tilde{B})_T$, for some $\FFF$-Brownian motion $\tilde{B}$, then, for all $\zeta\in\bmo$, we have
  \begin{equation}
  \label{equ:CM}
  (1+R)^{-1}\norm{\zeta}_{\bmo}\leq \norm{\zeta}_{\bmo(\hPP)}\leq (1-R)^{-1}\norm{\zeta}_{\bmo}.
  \end{equation}
\end{lem}

\begin{proof}
Since $M=\sigma \cdot \tilde{B}$ is a BMO-martingale, Theorem 3.6. in \cite{Kazamaki} states that the spaces $\bmo$ and $\bmo(\hPP)$ coincide and that the
norms $\norm{\cdot}_{\bmo}$ and $\norm{\cdot}_{\bmo(\hat{\PP})}$ are uniformly equivalent. This norm equivalence is refined in \cite{Chikvinidze-Mania}; Theorem 2 there implies that
  \begin{equation}\label{equ:CM-in}
  (1+R)^{-1}
  \norm{\zeta}_{\bmo} \leq \norm{\zeta}_{\bmo(\hat{\PP})} \leq
  (1+\hat{R}) \norm{\zeta}_{\bmo}, \text{ where } \hat{R} = \sqrt{\tot \norm{\sigma}^2_{\bmo(\hPP)}}.
 \end{equation}
 Clearly, only the second inequality in
\eqref{equ:CM}  needs to be discussed; it is obtained by
 substituting $\zeta=\sigma$ into the second inequality in \eqref{equ:CM-in}:
 \[  \sqrt{2} \hat{R} = \norm{\sigma}_{\bmo(\hPP)} = (1+\hat{R}) \norm{\sigma}_{\bmo} \leq \sqrt{2} (1+\hat{R}) R, \text{ so that } (1+\hat{R}) \leq (1-R)^{-1}.\qedhere\]
 \end{proof}


To prove the global uniqueness of equilibrium, we record the following
a-priori estimate on $\lambda$ in equilibrium. 
\begin{lem}\label{lem:lambda est}
There exists a universal constant $C$ such that for any equilibrium
$\ld\in\bmo$ 
 \[
  \|\lambda\|_{\bmo} \leq C \max_i \norm{(m^i,n^i)}_{\bmo}.\]
\end{lem}
\begin{proof}
Supposing that $\ld\in\bmo$ is an equilibrium,  
we subtract $X^i$ from $\Yil$,  sum over all $i$, and use 
 the second equation in \eqref{equ:system}, to obtain
 \begin{align*}
  \textstyle d \sum_i (\Yil_t - X_t^i) &= \textstyle \sum_i(\muil_t -m^i_t)\, dB_t + \sum_i
  (\nuil_t - n^i_t)\, dW_t + \\ &
  \quad + \textstyle\sum_i \big( h^i_t(\lambda_t, \nuil_t) - \lambda_t \partial_1
  h^i_t(\lambda_t, \nuil_t) - \tfrac{1}{2\delta_i}((m^i_t)^2+ (n^i_t)^2)\big)\, dt,
 \end{align*}
 where both stochastic integrals on the right-hand side are $\BMO$-martingales.
 Using Lemma \ref{lem:h} part (4), 
 the previous inequality and the fact that $\Yil \geq X^i$,  $\Yil_T=
 X^i_T$, we get 
 \[
  \tfrac{\gamma}{2} I \,\EE_\tau \Big[\int_\tau^T \lambda_u^2 du\Big] \leq
  \tfrac{1}{2\delta} {\textstyle\sum } \,\EE_\tau \Big[ \int_\tau^T (m^i_u)^2 +
  (n^i_u)^2 du\Big] \leq \tfrac{1}{2\delta} I \max_i \|(m^i, n^i)\|^2_{\bmo}. 
 \]
 for each stopping time $\tau$, 
 confirming the claim with $C=1/\sqrt{\delta \gamma}$. 
\end{proof}

For $\lambda\in \bmo$ close enough to $0$, the following estimate gives an
explicit upper bound on the (nonnegative) difference between $D^i = \Yil -
X^i$. In it, we set
 \begin{equation}
 \label{equ:rho}
    r(p) = \frac{ \sqrt{2}(M + p)}{ \sqrt{2}-p}, \quad \ewhere M= \max_i \norm{(m^i,n^i)}_{\bmo}.
 \end{equation}

\begin{lem}\label{lem: D}
There exists a universal constant $C$ such that
 \[ 0\leq \sqrt{D^i} \leq C   r (\norm{\ld}_{\bmo}), \quad \text{ for all } i \text{ and }
 \ld\in\bmo \text{ with } \norm{\ld}_{\bmo}<\sqrt{2}.\]
\end{lem}

\begin{proof}
  The variational definition of  $\Yil$ in \eqref{equ:def-Yl}  yields 
 \[
  \Yil_t \leq \EE^{\QQ^\lambda}_t [E^i] + \tot \Delta
  \|\lambda\|^2_{\bmo(\QQ^\lambda)}, \quad  \ewhere
 \tfrac{d\QQ^\lambda}{d\PP} = \mathcal{E}\big(-\int \lambda_u dB_u\big)_T.\]
With $Z$ denoting the density $Z_t=\EN(-\ld\cdot B)_t$, we have
 \[
  \EE^{\QQ^\lambda}_t[E^i] = \tfrac{1}{Z_t}\EE_t [Z_T E^i ] =
  \tfrac{1}{Z_t}\EE_t [Z_T X^i_T],
 \]
and It\^{o}'s formula implies that
 \[
  \tfrac{1}{Z_t}\EE_t [Z_T X^i_T] = X^i_t + \EE^{\QQ^\lambda}_t
  \Big[\int_t^T \tfrac{1}{2\delta_i}((m^i_u)^2+ (n^i_u)^2) -\lambda_u m^i_u \, du\Big].
 \]
 The previous estimates, combined  with $-\lambda m^i \leq
 \tfrac{\delta_i}{2} \lambda^2 + \tfrac{1}{2\delta_i}(m^i)^2$, produce a
 universal constant $C$ such that
 \[
  D^i_t \leq C \big(\|\lambda\|^2_{\bmo(\QQ^\lambda)} + \|(m^i, n^i)\|^2_{\bmo(\QQ^\lambda)}\big) \leq C \big(\|\lambda\|_{\bmo(\QQ^\lambda)} + \|(m^i, n^i)\|_{\bmo(\QQ^\lambda)}\big)^2,
 \]
 and the statement follows from Lemma \ref{lem:eq-bmo}.
\end{proof}

\begin{lem}\label{lem: mu est}
There exist universal constants $C$ and $\epsilon <\sqrt{2}$ such that 
\[   \|(\muil, \nuil)\|_{\bmo}  \leq  C(M+p^2), \quad \text{ for any } M, p \leq \epsilon,
\]
 where $M= \max_i \norm{(m^i, n^i)}_{\bmo}$ and $p=\norm{\ld}_{\bmo}$.
\end{lem}
\begin{proof}
 For $\ld$ with $p=\|\lambda\|_{\bmo}< \sqrt{2}$,  let $ r(p)$ be as in
 \eqref{equ:rho} and $\cd$ as in Lemma \ref{lem: D} so that
 $0\leq D^i \leq \cd^2  r^2(p)$ and $D^i_T=0$. Since
  \begin{multline*}
  dD^i_t = (\muil_t- m^i_t) dB_t + (\nuil_t - n^i_t) dW_t + \\ +
  \big[h^i_t(\lambda_t, \nuil_t)  
  + \muil_t \lambda_t - \tfrac{1}{2\delta_i}\big((m^i_t)^2+ (n^i_t)^2\big)\big] dt,
  \end{multline*}
 an application of It\^{o}'s formula yields
 \begin{align*}
  d(D^i_t)^2 = &2D_t^i (\muil_t-m^i_t) dB_t + 2D^i_t (\nuil_t - n^i_t) dW_t + \big[(\muil_t - m^i_t)^2 + (\nuil_t - n^i_t)^2\big] dt \\
  & + 2D^i_t \big[h^i_t(\lambda_t, \nuil_t)+\muil_t \lambda_t -
  \tfrac{1}{2\delta_i}\big((m^i_t)^2 + (n^i_t)^2\big)\big] dt.
 \end{align*}
 The stochastic integrals on the right-hand side are martingales, since
 $D^i$ is bounded. Using the fact that $D_T =0$, $D\geq 0$, $h^i(\lambda,
 \nuil) \geq -\tfrac{\Delta}{2} \lambda^2$, and $\muil \lambda \geq
 -\tfrac12 (\muil)^2 -\tfrac12 \lambda^2$, we conclude there exists a
 universal constant $C$ such that
 \begin{align*}
  \EE_\tau \Big[\int_\tau^T \big[(\muil -m^i)^2 + (\nuil-n^i)^2\big]dt
  \Big] \leq &C r^2(p) \big(\|\lambda\|^2_{\bmo} + \|\muil\|^2_{\bmo} + \|(m^i,n^i)\|^2_{\bmo}\big)\\
  \leq & C r^2(p) \big(\|\lambda\|_{\bmo} + \|\muil\|_{\bmo} + M\big)^2.
 \end{align*}
It remains to observe that
 \begin{align*}
  \|(\muil, \nuil)\|_{\bmo} - \|(m^i,n^i)\|_{\bmo} &\leq \|(\muil - m^i, \nuil-n^i)\|_{\bmo} \\
  &\leq C r(p) \big(\|\lambda\|_{\bmo} + \|(\muil,
  \nuil)\|_{\bmo} + M\big).
 \end{align*}
 When $1- Cr(p)>0$, i.e., $1-p/\sqrt{2} - C(M+p)>0$, rearranging the previous inequality yields
 \[
  \|(\mu^{i,\lambda}, \nu^{i,\lambda})\|_{\bmo} \leq \frac{(1-p/\sqrt{2}) M + C(M+p)^2}{1-p/\sqrt{2} - C(M+p)}.
 \]
 There exists a sufficient small universal costant $\epsilon$ such that, when $p, M\leq \epsilon$, we have $1-p/\sqrt{2} - C(M+p) \geq 1/2$, hence 
 \[
  \frac{(1-p/\sqrt{2}) M + C(M+p)^2}{1-p/\sqrt{2} - C(M+p)} \leq C(M+p^2),
 \]
 where the universal constant $C$ on the right-hand side may be different from the one on the left. The statement then follows from combining the previous two inequalities.  \qedhere
\end{proof}

Define $\sB_{\bmo}( p) = \{\lambda \in \bmo\,:\, \|\lambda\|_{\bmo}\leq
p\}$. The following result shows that the excess-demand map $F$ maps
$\sB_{\bmo}( p)$ into itself for an appropriate choice of $ p$, when
$\max_i \|(m^i, n^i)\|_{\bmo}$ is sufficiently small.

\begin{lem}\label{lem: excess demand}
There exist a universal constant $\eps$ and a universal function
$\bar{\eps}:[0,\eps) \to
(0,\infty)$ such that
\begin{enumerate}
\item  $\lim_{M\to 0} \bar{\eps}(M)=0$ and $\lim_{M\to 0} \bar{\eps}(M)/M > \cle$, where
$\cle$  is the
constant of Lemma \ref{lem:lambda est}, and
\item 
$F$ maps $\sB_{\bmo}(\bar{\eps}(M))$ into itself,
as soon as $M=\max_i \norm{(m^i,n^i)}_{\bmo}<\eps$.
\end{enumerate}
\end{lem}
\begin{proof}
Since $H^{-1}(0,0)= 0$,   Lemma \ref{lem: H} guarantees the existence of a
positive universal constant $\ch$ such that
\begin{equation}\label{eq:Flam}
\begin{split}
 \norm{F(\ld)}_{\bmo} = \norm{ H^{-1}(-\tfrac{1}{I} {\textstyle \sum_i} \mu^{i,\lambda}, \anul)}_{\bmo} &\leq
\ch \big(\max_i \|\mu^{i,\lambda}\|_{\bmo} + \max_i \|\nu^{i,\lambda}\|_{\bmo} \big),\\
& \leq 2 \ch \max_i \|(\mu^{i,\lambda}, \nu^{i,\lambda})\|_{\bmo},
\end{split}
\end{equation}
for all $\ld\in\bmo$.

  With $M=\max_i \norm{(m^i,n^i)}_i$, $p=\|\lambda\|_{\bmo}$, and $\cme$, $\eme$ 
  denoting the constants from Lemma \ref{lem: mu est}, we have
   \[ 
   \norm{F(\ld)}_{\bmo} \leq   2\ch \cme (M+p^2),
   \] 
   for any $p, M\leq \eme$. Choosing a universal constant $C$ larger than $2\ch \cme$ and $\cle$, we have from the previous inequality that 
   \[
    \norm{F(\ld)}_{\bmo} \leq C(M+p^2).
   \]
   There exists a universal constant $\eps_0\leq \eme$ such that the quadratic equation $f(p) := C(M+p^2) -p =0$ admits at least one solution, whenever $M\leq \eps_0$. Denote the smaller solution as $\bar{\eps}(M)$. The expression of $\bar{\epsilon}(M)$ yields $\lim_{M\to 0} \bar{\eps}(M) =0$. Mover the equation $f(p)=0$ implies $\liminf_{M\to 0} \tfrac{\bar{\eps}(M)}{M} =\liminf_{M\to 0} \tfrac{C(M + \bar{\eps}(M)^2)}{M} \geq C \geq \cle$. It is then easy to see that
   \[
    \|F(\lambda)\|_{\bmo} \leq C(M + \bar{\eps}(M)^2) = \bar{\eps}(M),
   \]
   for any $\lambda$ with $\|\lambda\|_{\bmo} \leq \bar{\eps}(M)$. 
\end{proof}

\begin{lem}
\label{lem:67DB}
There exists universal constants $\eps,C$ such that, 
if   $\max_i \norm{(m^i,n^i)}_{\bmo}\leq \eps$, then for any $\lambda, \tilde{\lambda}$ satisfying $\|\lambda\|_{\bmo}, \|\tilde{\lambda}\|_{\bmo} \leq \bar{\eps}(\max_i \|(m^i, n^i)\|_{\bmo})$, we have 
\[ \norm{ F(\ld) - F(\tld)  }_{\bmo} \leq
C(L^{\ld}+L^{\tld})
 \norm{\ld - \tld}_{\bmo},\]
where
   $L^{\ld}= \norm{\ld}_{\bmo} +
   \max_i\norm{ (\muil,\nuil)}_{\bmo}$ and
   $L^{\tld}= \norm{\tld}_{\bmo} +
   \max_i\norm{ (\muilt,\nuilt)}_{\bmo}$.
\end{lem}
\begin{proof}
In the first part of the proof we suppress the index $i$ notationally, as
we will be focusing on a single-agent $\Yl$. 
For $\ld,\tld\in\bmo$ with $\norm{\ld}_{\bmo},
\norm{\tld}_{\bmo}<\sqrt{2}$ and denote by $(Y,\mu,\nu)=(\Yl,\mul,\nul)$ and
$(\tY,\tmu,\tnu)=(Y^{\tld}, \mu^{\tld},\nu^{\tld})$,
the corresponding solutions to \eqref{eq: BSDE Y}. By the argument in the
proof of Lemma
\ref{lem: D}, helped by that fact that it terminates at $0$, the process $\delta Y = Y- \tY$ belongs to
$\sS^{\infty}$. 
 We set $\overline{\lambda} =
 (\lambda + \tilde{\lambda})/2$ and $\overline{\mu} = (\mu +
 \tilde{\mu})/2$ so that 
 \begin{align*}
  d\delta Y_t &= (\mu_t -\tilde{\mu}_t) dB_t + (\nu_t - \tilde{\nu}_t) dW_t + \big(h_t(\lambda_t, \nu_t) - h_t(\tilde{\lambda}_t, \tilde{\nu}_t) + \mu_t \lambda_t - \tilde{\mu}_t \tilde{\lambda}_t\big) dt\\
  &= (\mu_t -\tilde{\mu}_t) dB^{\overline{\lambda}}_t + (\nu_t - \tilde{\nu}_t) dW_t^{\overline{\nu}} + \big(h_t(\lambda_t, \nu_t) - h_t(\tilde{\lambda}_t, \nu_t) + \overline{\mu}_t (\lambda_t - \tilde{\lambda}_t)\big) dt.
 \end{align*}
 Here
 \[
  \overline{\nu} = \left\{\begin{array}{ll} \frac{h(\tilde{\lambda}, \nu) - h(\tilde{\lambda}, \tilde{\nu})}{\nu - \tilde{\nu}} & \nu \neq \tilde{\nu} \\ 0 & \text{otherwise}\end{array}\right.,
 \]
 which satisfies
 \begin{equation}\label{eq:bar nu}
  \|\bar{\nu}\|_{\bmo} \leq \Theta (\|\tilde{\lambda}\|_{\bmo} + \|\nu\|_{\bmo} + \|\tilde{\nu}\|_{\bmo}),
 \end{equation}
 thanks to \eqref{equ:h-lip}. Then $B^{\overline{\lambda}} = B + \int_0^\cdot \overline{\lambda}_t\,
 dt$, $W^{\overline{\nu}} = W + \int_0^ \cdot \overline{\nu}_t\, dt$ are 
 Brownian motions under $\QQ^{\overline{\lambda}, \overline{\nu}}$. Utilizing \eqref{equ:h-lip} again, there exists a universal constant $C$ such that
 \[
  |\delta Y_t| \leq C \,\EE_t^{\QQ^{\overline{\lambda},
  \overline{\nu}}}\Big[\int_t^T \big(|\lambda_u| + |\tilde{\lambda}_u| +
  |\nu_u| + |\overline{\mu}_u|\big) |\lambda_u -\tilde{\lambda}_u| \,
  du\Big],
 \]
 and the Cauchy-Schwartz inequality yields
 \begin{align}\label{deltaY est}
 \|\delta Y\|_{\mathcal{S}^\infty} \leq  C
 (\bar{L}^{\ld}+\bar{L}^{\tld}) \|\lambda
 -\tilde{\lambda}\|_{\bmo(\QQ^{\overline{\lambda}, \overline{\nu}} )}.
 \end{align}
 where $\bar{L}^{\ld}$ and $\bar{L}^{\tld}$ are analogous to $L^{\ld}$ and
 $L^{\tld}$, but with the $\bmo$-norms computed under 
  $\QQ^{\overline{\lambda}, \overline{\nu}}$. 
Next, by 
 It\^{o}'s formula, we have
 \begin{multline*}
  d(\delta Y_t)^2 =  2 \delta Y_t (\mu_t -\tilde{\mu}_t)
  dB_t^{\overline{\lambda}} + 2 \delta Y_t (\nu_t - \tilde{\nu}_t)
  dW_t^{\overline{\nu}} + \big((\mu_t- \tilde{\mu}_t)^2 + (\nu_t -
  \tilde{\nu}_t)^2\big) dt + \\
  + 2 \delta Y_t \big(h_t(\lambda_t, \nu_t) - h_t (\tilde{\lambda}_t, \nu_t) +
  \overline{\mu}_t(\lambda_t - \tilde{\lambda}_t)\big) dt, 
 \end{multline*}
 so that, thanks to \eqref{equ:h-lip}, 
 \eqref{deltaY est}, and the fact that $\delta Y_T=0$ we obtain
 \begin{align*}
  \EE^{\QQ^{\overline{\lambda}, \overline{\nu}}}\Big[\int_\tau^T
 (\mu_u-\tilde{\mu}_u)^2 + (\nu_u - \tilde{\nu}_u)^2 \,du\Big] &\leq C
 \|\delta Y\|_{\mathcal{S}^\infty} (\bar{L}^{\ld}+\bar{L}^{\tld})
 \|\lambda
 -\tilde{\lambda}\|_{\bmo(\QQ^{\overline{\lambda}, \overline{\nu}} )}\\
 &\leq C  (\bar{L}^{\ld}+\bar{L}^{\tld})^2
 \|\lambda
 -\tilde{\lambda}\|_{\bmo(\QQ^{\overline{\lambda}, \overline{\nu}} )}^2,
 \end{align*}
 for any stopping time $\tau$.
 This, in turn, implies 
 \begin{equation}\label{eq:con-est}
  \|(\mu, \nu) - (\tilde{\mu},
  \tilde{\nu})\|_{\bmo(\QQ^{\overline{\lambda}, \overline{\nu}})}  \leq  C
  (\bar{L}^{\ld}+\bar{L}^{\tld}) \|\lambda
  -\tilde{\lambda}\|_{\bmo(\QQ^{\overline{\lambda}, \overline{\nu}} )}.
\end{equation}

The definition of $F$ and Lemma \ref{lem: H} imply that it will be enough to 
replace all $\bmo$-norms under
$\QQ^{\overline{\ld},\overline{\nu}}$ in \eqref{eq:con-est} above, as well
as in the expression for $\bar{L}^{\ld}$, $\bar{L}^{\tld}$, by those under $\PP$, perhaps after
enlarging the universal constant $C$. To do that, for $M= \max_i \|(m^i, n^i)\|_{\bmo} \leq \eed$, where $\eed$ is the universal constant in Lemma \ref{lem: excess demand}, take any $\|\lambda\|_{\bmo}, \|\tilde{\lambda}\|_{\bmo}\leq \bar{\eps}(M)$. Lemma \ref{lem: mu est} implies that 
\begin{equation}\label{eq: mu tmu}
 \|(\mu, \nu)\|_{\bmo}, \|(\tilde{\mu}, \tilde{\nu})\|_{\bmo} \leq C (M + \bar{\epsilon}(M)^2).
\end{equation}
Set $R = \tfrac{1}{\sqrt{2}} \|(\bar{\lambda}, \bar{\nu})\|_{\bmo}$. Combining \eqref{eq:bar nu}, \eqref{eq: mu tmu}, and $\lim_{M\to 0} \bar{\eps}(M) =0$, we can choose sufficiently small $\eps$ so that $R<1$ when $M\leq \epsilon$. Then applying Lemma \ref{lem:eq-bmo} to both sides of \eqref{eq:con-est}, we obtain
\begin{equation}\label{eq:con-est-P}
 \|(\mu, \nu) - (\tilde{\mu}, \tilde{\nu})\|_{\bmo} \leq C \frac{1+R}{(1-R)^2} (L^\lambda + L^{\tilde{\lambda}}) \|\lambda - \tilde{\lambda}\|_{\bmo}.
\end{equation}
Finally, Lemma \ref{lem: H} and an estimate similar to \eqref{eq:Flam} imply
\[
 \|F(\lambda) - F(\tilde{\lambda})\|_{\bmo} \leq 2 \ch \max_i \|(\mu^i, \nu^i) - (\tilde{\mu}^i, \tilde{\nu}^i)\|_{\bmo}. 
\]
The proof is concluded after combining and last two inequalities and reintroducing the index $i$ to the left-hand side of \eqref{eq:con-est-P}.
\end{proof}

\begin{proof}[Proof of Theorem \ref{thm: main}]
We pick the constant $M$ sufficiently small that $\bar{\eps}(M)$ of Lemma \ref{lem:
excess demand} is well defined, and  has the 
following properties:
\begin{enumerate}
  \item $\cle\, M \leq  \bar{\eps}(M)$, where $\cle$ is as in Lemma
  \ref{lem:lambda est}.
  \item when $\norm{\ld}_{\bmo}\leq \bar{\eps}(M)$, we have
   $L^{\ld} \leq \tfrac{1}{3\clip}$, where $L^{\ld}$ and $\clip$ are as in Lemma \ref{lem:67DB}, 
\end{enumerate}
Item (1) can be achieved thanks to Lemma
\ref{lem: excess demand} item (1), and item (2) can be satisfied thanks to \eqref{eq: mu tmu} and Lemma \ref{lem: excess demand} item (1). 

Assuming that $\max_i\norm{(m^i,n^i)}_{\bmo}\leq M$, Lemma \ref{lem: excess
demand} implies that $F$ maps $\sB_{\bmo}(\bar{\eps}(M))$ into itself. Moreover, item
(2) above and Lemma \ref{lem:67DB} imply that $F$ is a contraction on
$\sB_{\bmo}(\bar{\eps}(M))$. Therefore, by the Banach fixed point theorem, $F$
admits a unique fixed point in $\sB_{\bmo}(\bar{\eps}(M))$. This implies
immediately that the system  \eqref{equ:system} admits
a solution $(\aY, \amu, \anu)$ with $(\amu, \anu)\in \bmo^{2I}$, making $\lambda$ an equilibrium
 by Theorem \ref{thm:char}.  
 
Turning to uniqueness, Lemma \ref{lem:lambda est} implies that any
equilibrium needs to be in the ball of radius $ \cle M$, which is less than $\bar{\eps}(M)$ due to item (1) above. We have already established the uniqueness
of equilibria in $\sB_{\bmo}(\bar{\eps}(M))$, so the equilibrium $\ld$,  as well
as the associated solution $(\aY, \amu,\anu)$ of \eqref{equ:system}
constructed above, are globally unique. 
\end{proof}

\subsection{Proof of Corollary \ref{cor: homo}}
We sum both sides of 
of $\norm{E^i- E^j}_{\linf} \leq \chi^E_0 (\norm{E^i}_{\linf} +
\norm{E^j}_{\linf})$ over $j$ to obtain
\begin{multline*}
 I \norm{E^i}_{\linf} - \norm{E_\Sigma}_{\linf} \leq \|I E^i - \tsum_j E^j\|_
 {\linf} \leq \tsum_j \norm{E^i - E^j}_{\linf} \leq \\
  \leq \chi^E_0 I \norm{E^i}_{\linf} + \chi^E_0 \tsum_j \norm{E^j}_{\linf},
\end{multline*}
which implies that
\[
 (1- \chi^E_0) \norm{E^i}_{\linf} \leq \tfrac{1}{I} \norm{E_\Sigma}_{\linf} + \chi^E_0 \tfrac{1}{I} \tsum_j \norm{E^j}_{\linf}.
\]
Summing the obtained inequalities over $i$, we get
\[
 \tsum_i \norm{E^i}_{\linf} \leq \tfrac{1}{1-2 \chi^E_0} \norm{E_\Sigma}_{\linf}.
\]
The previous two inequalities combined then imply
\[
 \norm{E^i}_{\linf} \leq \tfrac{1}{1-2 \chi_0^E} \tfrac{1}{I} \norm{E_\Sigma}_{\linf}, \text{ for all } i.
\]
On the other hand, Proposition \ref{prop:ebmo} part (4) implies that
\[
 \|(m^i, n^i)\|^2_{\bmo} \leq 4 \delta^i \|E^i\|_{\linf} \leq 4 \Delta \|E^i\|_{\linf} \leq \tfrac{4\Delta}{1-2\chi^E_0} \tfrac{1}{I} \|E_\Sigma\|_{\linf}, \quad \text{ for all } i.
\]
Then right-hand side is smaller than $M$ in \eqref{equ:mn-bmo} for $I$ larger than some $I_0$, and the existence of equilibrium follows from Theorem \ref{thm: main}.

\appendix

\section{Characterization and properties of $\EBMO$}\label{app: EBMO}

The entropic $\BMO$ space introduced in Definition \ref{def:1F98} can be characterized via the reverse H\" older inequality, which is 
equivalent to the membership in $\BMO$; cf.
\cite[Theorem 3.4]{Kazamaki}. 
\begin{prop} \label{pro:ebmo-char}
The random variable $E$ is in $\EBMO$ 
if and only if $e^{-E} \in \lone$ and 
there exist constants $p>1$ and  $C>0$
such that for each stopping time $\tau$, we have
\[ \EE[ e^{-pE}|\sF_{\tau}] \leq C (\EE[e^{-E}|\sF_{\tau}])^{p}.\]
\end{prop}
Somewhat weaker statements in the following result will, perhaps, shed more light on the structure of
$\EBMO$:
\begin{prop}  \label{prop:ebmo}
The following hold:
\begin{enumerate}
\item If $E\in\EBMO$, then $E/\alpha \in \EBMO$ for each $\alpha>1$, 
\item If $E\in\EBMO$ then  $e^{-E} \in \cup_{p>1} \lpee$.
\item If $H\in\BMO$ is positive and bounded away from $0$, then $\log H \in
\EBMO$. 
\item $\linf \subseteq \EBMO$; in fact, if $E\in\linf$ then $\norm{E}_{\EBMO} \leq 2 \norm{E}^{1/2}_{\linf}$. 
\end{enumerate}
\end{prop}
\begin{proof} (1) and (2) follow directly from Proposition \ref{pro:ebmo-char}.
For (3), we note that the strictly positive BMO-martingale $h_t=\EE_t[H]$ admits a stochastic logarithm $M_t = \int_0^t h_t^{-1}\, dh_t$. Moreover,
since $h$ is bounded away from zero, the quadratic variation of $M$ is bounded
from above by a constant multiple of the quadratic variation of $h$, and,
so, $M\in\BMO$, i.e., $\log H\in\EBMO$. 
The fact that  $\linf \subseteq \EBMO$ is a direct consequence of the fact that $\linf\subseteq\BMO$. Furthermore, let $N$ be the 
continuous martingale given by $N_t = \EE_t[ e^{-E}]$.
Since $N$ is an $\ltwo$-martingale bounded away from zero, the process $M$ defined via $M = - \int_0^\cdot  N_u^{-1}\, dN_u$, so that $\log N = \log N_0
- M - (1/2) \ab{M}$, is also an $\ltwo$
martingale. Moreover, we have
\begin{align*}
\tot \EE_t [ \ab{M}_T - \ab{M}_t] &= \EE_t[ (\tot \ab{M}_T + M_T) -  
(\tot \ab{M}_t + M_t)] \\ & = \EE_t[ \log(N_t/N_T)]\leq 2 \norm{E}_{\linf},
\end{align*}
and $\norm{E}_{\EBMO} \leq 2 \norm{E}^{1/2}_{\linf}$ follows directly from the fact that $\norm{E}_{\EBMO} = \norm{M}_{\BMO}$.
\end{proof}

Before we give another useful sufficient condition for membership in $\EBMO$, let us
recall briefly the notion of Malliavin differentiation on Wiener space. Let $\Phi$ be the set of random variables of the form $\varphi(\mathcal{I}(\eta^1), \dots,
\mathcal{I}(\eta^k))$, where $\varphi \in C^\infty_b(\mathbb{R}^k, \mathbb{R})$
(smooth functions with bounded derivatives of all orders) for some $k$,
$\eta^j=(\eta^{j,b}, \eta^{j,w})\in \mathbb{L}^2([0,T]; \mathbb{R}^2)$  and
$\mathcal{I}(\eta^j)= \eta^{j,b}\cdot B_T + \eta^{j,w}\cdot W_T$, for each $j=1, \dots,
k$. If $\zeta = \varphi(\mathcal{I}(\eta^1), \dots,
\mathcal{I}(\eta^k)) \in \Phi$, we define its \define{Malliavin derivative} as the
$2$-dimensional process
\[
 D_\theta \zeta = \sum_{j=1}^k \frac{\partial \varphi}{\partial x_j}
 (\mathcal{I}(\eta^1), \dots, \mathcal{I}(\eta^k)) \eta^j_\theta, \quad
 \theta\in [0,T],
\]
and denote by $D^b \zeta$ and $D^w \zeta$ the two component processes of 
$D \zeta$. For  $\zeta\in \Phi$ and  $p\geq 1$, we define the norm
\[
 \norm{\zeta}_{1,p}= \bra{\EE\bra{|\zeta|^p + \pare{\int_0^T |D_\theta
 \zeta|^2 d\theta}^{p/2}}}^{1/p},
\]
and let the Banach space $\mathbb{D}^{1,p}$ be the closure of
$\Phi$ under $\norm {\cdot}_{1,p}$. We say that a random variable $E$ is
\define{Malliavin-Lipschitz} if $E\in \bD^{1,2}$ and
$D^b E, D^w E\in \sinf$. The constant
\[ L=\Bnorm{\sqrt{\abs{D^b E}^2+\abs{D^w E}^2}}_{\sS^{\infty}}\]
is called the \define{Lipschitz constant} of $E$. 

In a Markovian setting, where $E =g(B_T, W_T)$, for some function $g$, $E$ is Malliavin-Lipschitz whenever $g$ is a Lipschitz function, and the Lipschitz constant of $E$ is the Lipschitz constant of $g$.

\begin{prop}  \label{prop:suf-ebmo}
If $E$ is Malliavin-Lipschitz with the Lipschitz constant $L$ then 
$E/\delta \in\EBMO$ and $\norm{E}_{\EBMO,\delta}\leq L \sqrt{T}$, for each
$\delta>0$. 
\end{prop}

\begin{proof} 
By the Clark-Ocone formula the components $\om$ and $\on$ in the 
 martingale representation $E = \EE[E] + \bar M_T = \EE[E] + \om\cdot B_T+\on\cdot W_T$ satisfy
\[  \om_t = \EE_t[ D^b E] \text{ and } \on_t = \EE_t[ D^w E], \text{ a.s.,
for each } t\in [0,T],\]
and, therefore, admit versions with $\sqrt{(\om_t)^2+(\on_t)^2} \leq L$, for
each $t\in [0,T]$, a.s. 
 As a result, $\ab{\bar M}_T\leq L$ and 
 Bernstein inequality (see Equation (4.i) in \cite{Barlow-et-al}),
 implies that $E$ has (at most) Gaussian tails. In particular, $e^{-
 E}\in\ltwo$. Coupled with the boundedness of the Malliavin derivatives of
 $E$, this fact implies that
$e^{-E} \in \bD^{1,2}$ and, consequently, with equalities interpreted in
the sense of modifications, 
  \begin{align*} V_t = \EE_t[ e^{-E}]\in\bD^{1,2} \eand D^k_\theta V_t = -
  \EE_t[e^{-E} D^k_\theta E] 
  \end{align*} 
  for all $\theta\leq t\leq T$ and
  $k=b$ or $w$.
  Applying
     Clark-Ocone formula to $V_t$ yields \[ V_t = \EE[V_t] + \int_0^t
     \EE_\theta[D^b_\theta V_t] dB_\theta + \int_0^t \EE_\theta[D^w_\theta
     V_t] dW_\theta.  \] 
     On the other hand, $dV_\theta = - V_\theta
     m_\theta dB_\theta - V_\theta n_\theta dW_\theta$, and, so, 
     $\EE_\theta[D^b_\theta V_t] = -V_\theta m_\theta$ and $\EE_\theta
     [D^w_\theta V_t] = -V_\theta m_\theta$, for $\theta \leq t$. Hence, \[
     m_\theta = - \frac{\EE_\theta[D^b_\theta V_t]}{V_\theta} =
     \frac{\EE_\theta[e^{-E} D^b_\theta E]}{\EE_\theta[e^{-E}]}\leq
     \norm{D^b E}_{\sinf}, \] which implies $\norm{m}_{\sinf}\leq \norm{D^w
     E}_{\sinf}$.  Similarly, $\norm{n}_{\sinf}\leq \norm{D^w
     E}_{\sinf}$, and the bound in (2) follows immediately.\qedhere
 \end{proof}

\bibliographystyle{amsalpha}
\bibliography{./biblio}

\end{document}